\documentclass[12pt]{amsart}
\usepackage{amsmath,amsthm,amsfonts,amssymb,verbatim,eucal}

\numberwithin{equation}{section}

\newtheorem{thm}[equation]{Theorem} 
\newtheorem{prop}[equation]{Proposition}
\newtheorem{lemma}[equation]{Lemma} 
\newtheorem{cor}[equation]{Corollary}
\newtheorem{example}[equation]{Example}
\newtheorem{remark}[equation]{Remark}
\newtheorem{definition}[equation]{Definition}

\DeclareMathOperator{\Alt}{Alt}
\DeclareMathOperator{\Ext}{Ext}
\DeclareMathOperator{\gr}{gr} 
\DeclareMathOperator{\Ima}{Im}
\DeclareMathOperator{\sgn}{sgn}
\DeclareMathOperator{\Sym}{Sym}

\newcommand{\RR}{\mathbb R}
\newcommand{\DOT}{\setlength{\unitlength}{1pt}\begin{picture}(2.5,2)
               (1,1)\put(2,3.5){\circle*{3}}\end{picture}}

\newcommand{\Z}{{\mathbb Z}}

\newcommand{\id}{\mbox{\rm id\,}}      
\newcommand{\Hom}{\mbox{\rm Hom\,}}

\newcommand{\F}{\mathcal F}

\renewcommand{\ker}{\mbox{\rm Ker\,}}

\newcommand{\ot}{\otimes}

\newcommand{\cH}{\mathcal{H}}
\newcommand{\CC}{\mathbb{C}}  
\newcommand{\ZZ}{\mathbb{Z}}
\newcommand{\cF}{\mathcal{F}}
\newcommand{\cB}{\mathcal{B}}
\newcommand{\tensor}{\otimes}
\newcommand{\g}{{\mathfrak{g}}}

\DeclareMathOperator{\codim}{codim}
\newcommand{\HHD}{{\rm HH}^{\DOT}}
\newcommand{\HH}{{\rm HH}}
\newcommand{\Wedge}{\textstyle\bigwedge}
\newcommand{\HD}{H^{\hspace{.2ex}\DOT}}
\newcommand{\CD}{C^{\hspace{.3ex}\DOT}}

\begin{document}
\begin{abstract}
We define Drinfeld orbifold algebras as filtered algebras deforming
the skew group algebra (semi-direct product)
arising from the action of a finite group
on a polynomial ring.  They simultaneously generalize 
Weyl algebras, graded (or Drinfeld)
Hecke algebras, rational Cherednik algebras, symplectic reflection
algebras, and universal enveloping algebras of Lie algebras with group actions.
We give necessary and sufficient conditions on defining parameters
to obtain Drinfeld orbifold algebras in two general formats, both
algebraic and homological.  We
explain the connection between Hochschild cohomology and 
a Poincar\'e-Birkhoff-Witt property explicitly (using Gerstenhaber brackets).  
We also classify those
deformations of skew group algebras which arise
as Drinfeld orbifold algebras and give applications for abelian groups. 
\end{abstract}
\title[Drinfeld orbifold algebras]
{Drinfeld orbifold algebras}

\date{November 30, 2011}
\author{A.\ V.\ Shepler}
\address{Department of Mathematics, University of North Texas,
Denton, Texas 76203, USA}
\email{ashepler@unt.edu}
\author{S.\  Witherspoon}
\address{Department of Mathematics\\Texas A\&M University\\
College Station, Texas 77843, USA}\email{sjw@math.tamu.edu}
\thanks{The first author was partially supported by NSF grants
\#DMS-0800951 and \#DMS-1101177
and a research fellowship from the 
Alexander von Humboldt Foundation. 
The second author was partially supported by
NSF grants 
\#DMS-0800832 and \#DMS-1101399.}


\maketitle

\section{Introduction}
Results in commutative algebra are often obtained by
an excursion through a larger, noncommutative universe.
Indeed, interesting noncommutative algebras often arise from deforming
the relations of a classical commutative algebra.
Noncommutative algebras modeled on groups acting on
commutative polynomial rings serve as useful tools
in representation theory and combinatorics, for example,
and include symplectic reflection algebras, rational Cherednik algebras,
and Lusztig's graded affine Hecke algebras.
These algebras are deformations of the skew group algebra generated
by a finite group and a polynomial ring (upon which the group acts).
They also provide an algebraic framework for understanding geometric
deformations of orbifolds.

Let $G$ be a finite group acting by linear transformations on a 
finite dimensional vector space $V$  over a field $k$.
Let $S:=S(V)$ be the symmetric algebra with the induced action
of $G$ by automorphisms, 
and let $S\# G$ be the corresponding skew group algebra.
A graded Hecke algebra (often called a Drinfeld Hecke algebra) 
emerges after deforming the relations of the symmetric algebra $S$
inside $S\# G$: We set each expression 
$x y-y x$ (for $x,y$ in $V$) in the tensor algebra $T(V)$ 
equal to an element of the group ring $kG$
and consider the quotient of $T(V)\# G$ by such relations.
The quotient is a deformation of $S\# G$ when the relations satisfy
certain conditions, and these conditions are explored in many papers
(see, e.g.,~\cite{Drinfeld,EtingofGinzburg,Lusztig,RamShepler}).

Symplectic reflection algebras are special cases of graded Hecke
algebras which generalize
Weyl algebras in the context of group actions on symplectic spaces.
In this paper,
we replace Weyl algebras with universal enveloping algebras of Lie
algebras and complete the analogy: {\em Weyl algebras are to symplectic 
reflection algebras as universal enveloping algebras are to what?}
Our answer is the class of Lie orbifold algebras, which together with
graded Hecke algebras belong to a larger class of Drinfeld
orbifold algebras as we define and explore in this article.

In a previous article~\cite{SheplerWitherspoon1}, we explained
that graded Hecke algebras are precisely those deformations 
of $S\# G$ which arise from Hochschild 2-cocycles
of degree zero with respect to a natural grading
on cohomology.  In fact, we showed that every such cocycle defines a graded 
Hecke algebra and thus lifts to a deformation of $S\# G$.
The present investigation 
is partly motivated by a desire to understand deformations
of $S\# G$ arising from Hochschild 2-cocycles of degree one.

Specifically, we assign degree 1 to each $v$ in $V$ and
degree $0$ to each $g$ in $G$ and consider the corresponding grading on 
$T(V)\# G$.  
We set each expression $x\otimes y-y\otimes x$ in the 
tensor algebra $T(V)$ equal to an 
element of degree at most 1 (i.e., nonhomogenous of filtered degree $1$)
and consider the quotient of $T(V)\# G$
by these relations as a filtered algebra.  
We call the resulting algebra a 
{\em Drinfeld orbifold algebra} if it 
satisfies the Poincar\'e-Birkhoff-Witt property,
i.e., if its associated graded algebra is isomorphic to $S\# G$.
Such algebras were studied by Halbout, Oudom, and Tang~\cite{HOT}
over the real numbers in the special case that
$G$ acts faithfully. We give a direct algebraic approach for arbitrary
group actions and fields here.

In this article, we explain in detail the connections between 
the Poincar\'e-Birkhoff-Witt property, deformation theory,
and Hochschild cohomology.  
We first classify those deformations of $S\# G$ which arise
as Drinfeld orbifold algebras.  
We then derive necessary and sufficient conditions on algebra parameters
that should facilitate efforts to study and classify these algebras. 
In particular, we express the PBW property as a set of conditions
using the Diamond Lemma.  
(Although our conditions hold over arbitrary characteristic, 
we include a comparison with   
the theory of Koszul rings over $kG$ used in~\cite{EtingofGinzburg}
and~\cite{HOT}, which requires $kG$ to be semisimple.)
We give an explicit road map from cohomology, expressed in terms
of Koszul resolutions, to the defining relations for Drinfeld orbifold algebras.
In particular, we explain how PBW conditions enjoy an elegant description
in terms of Gerstenhaber brackets.
 
Note that one can not automatically deduce results 
for Drinfeld orbifold algebras 
defined over $\CC$ from the results in~\cite{HOT} for similar
algebras defined over $\RR$.
(For example, the infinitesimal 
of a nontrivial deformation over $\CC$ associated to a
group $G$ acting on a complex vector space
is always supported off the set ${\mathcal R}$ of complex reflections in $G$,
yet the infinitesimal of a nontrivial deformation over $\RR$ associated
to that same group (acting on a real vector space of
twice the dimension) may have support including ${\mathcal R}$.)

More precisely, let us consider a linear ``parameter'' function
mapping the exterior product $V\wedge V$ to that part of $T(V)\# G$
having degree at most 1:
 $$\kappa: V\wedge V \rightarrow (k \oplus V)\otimes kG\ .$$
We drop the tensor sign when expressing elements
of $T(V)$ and $T(V)\# G$
(as is customary when working with noncommutative, 
associative algebras), writing $vw$ in place of  $v\tensor w$,
for example.  We also usually write $\kappa(v,w)$ for
$\kappa(v\wedge w)$, to make some complicated expressions clearer.
Define an algebra $\cH:=\cH_{\kappa}$ as the quotient
$$\cH_{\kappa} := T(V) \# G   /    
(   
v  w - w  v  - \kappa(v, w)\, \mid v,w \in V).
$$
We say that $\cH_{\kappa}$ satisfies the {\bf PBW condition} when 
its associated graded algebra $\gr \cH_{\kappa}$ 
is isomorphic to $S\# G$ (in analogy with the 
Poincar\'e-Birkhoff-Witt Theorem for
universal enveloping algebras).  In this case, 
we call $\cH_{\kappa}$ a {\bf Drinfeld orbifold algebra}.
One may check that the PBW condition is equivalent
to the existence of a 
basis $\{v_1^{m_1}\cdots v_n^{m_n} g:m_i\in \Z_{\geq 0}, g\in G\}$
for $\cH_{\kappa}$ as a $k$-vector space,
where $v_1,\ldots, v_n$ is a $k$-basis of $V$.

The terminology arises because Drinfeld~\cite{Drinfeld} 
first considered deforming
the algebra of coordinate functions $S^G$ of the
orbifold $V^*/G$ (over $\CC$)
in this way, although his original construction 
required the image of $\kappa$ to lie in the group algebra $\CC G$.
Indeed, when $\kappa$ has image in $kG$, 
a Drinfeld orbifold algebra $\cH_{\kappa}$ is called 
a {\bf Drinfeld Hecke algebra}.
These algebras are also called
{\bf graded Hecke algebras}, as the graded affine Hecke algebra
defined by Lusztig~\cite{Lusztig88,Lusztig} 
is a special case (arising when $G$ is a Coxeter
group, see~\cite[Section~3]{RamShepler}).  Note that 
symplectic reflection algebras are also examples of these algebras.
 
Drinfeld orbifold algebras compose a large class of deformations
of the skew group algebra $S\#G$, as explained in this paper.
We determine necessary and sufficient
conditions on $\kappa$ so that $\cH_{\kappa}$ satisfies the PBW condition
and interpret these conditions in terms of Hochschild cohomology.
To illustrate, we give several small examples in Sections~\ref{sec:nands}
and \ref{sec:loa}.
We show that a special case of this construction is a class of
deformations of the skew group algebras ${\mathcal U}\# G$, where $\mathcal{U}$
is the universal enveloping algebra of a finite dimensional Lie algebra
upon which $G$ acts.
These deformations are termed {\bf Lie orbifold algebras}.

For example, consider
the Lie algebra ${\mathfrak{sl}}_2$
of $2\times 2$ matrices over $\CC$ having trace $0$
with usual basis $e,f,h$. A cyclic group $G$ of order~$2$ generated
by $g$ acts as follows: $ {}^g e = f, \ {}^gf = e, \ {}^g h=-h$.
Let $V$ be the underlying $\CC$-vector space of ${\mathfrak{sl}}_2$
and consider the quotient 
$$
   T(V)\# G / (eh-he+2e -g, \ hf-fh+2f-g, \ ef-fe-h).
$$
We show in Example~\ref{example:sl2} that this quotient 
is a Lie orbifold algebra.
Notice that if we delete the degree 0 term (that is, the group element $g$)
in each of the first two relations above, we obtain the skew group algebra 
${\mathcal{U}}({\mathfrak{sl}}_2)\# G$. If we delete
the degree 1 terms
instead, we obtain a Drinfeld Hecke algebra (i.e., graded Hecke algebra).
(This is a general property of Lie orbifold algebras that 
we make precise 
in Proposition~\ref{PBW}.)

We assume throughout that $k$ is a field whose characteristic is not 2.
For our homological results in Sections \ref{sec:koszul} through
\ref{sec:abel}, 
we require in addition that the order of $G$ is invertible in $k$ and that
$k$ contains all eigenvalues of the actions of elements of $G$ on $V$;
this assumption is not needed for the first few sections.
All tensor products will be over $k$ unless otherwise indicated.

\section{Deformations of Skew Group Algebras}
\label{deformations}

Before exploring necessary and sufficient conditions 
for an arbitrary
quotient algebra to define a Drinfeld orbifold algebra, we 
explain the connection between these algebras and
deformations of the skew group algebra $S\#G$.
Recall that $S\#G$ is the $k$-vector space $S\ot kG$
with algebraic structure given by
$(s_1\otimes g)(s_2\otimes h)=s_1\, ^g(s_2)\otimes gh$
for all $s_i$ in $S$ and $g,h$ in $G$. Here, $^gs$
denotes the element resulting from the 
group action of $g$ on $s$ in $S$.  Recall that we drop
the tensor symbols and simply write, for example, 
$s_1gs_2h=s_1\, ^gs_2 gh$. 
We show in the next theorem how Drinfeld orbifold algebras arise
as a special class of deformations of $S\# G$.

First, we recall some standard notation.
Let $R$ be any algebra over the field $k$, and let $t$ be an indeterminate.
A {\bf deformation of $R$ over $k[t]$} is an associative $k[t]$-algebra
with underlying vector space $R[t]$ and multiplication  determined by 
$$
  r*s = rs + \mu_1(r\ot s)t + \mu_2(r\ot s) t^2 + \cdots
$$
for all $r,s\in R$, where $rs$ is the product of $r$ and $s$ in $R$,
the $\mu_i: R\ot R\rightarrow R$ are $k$-linear maps that are extended to
be $k[t]$-linear, and 
the above sum is finite for each $r,s$.

We adapt our definition of $\cH_{\kappa}$ to that of an algebra over $k[t]$.
First, decompose $\kappa$ into its constant and linear parts:
Let $$
\kappa= \kappa^C+\kappa^L
\quad\text{where}\quad
\kappa^C:V\wedge V\rightarrow kG,\ \ \ 
\kappa^L:V\wedge V\rightarrow V\ot kG\ .
$$
Write 
$$\kappa=\sum_{g \in G} \kappa_g\ g $$
where each (alternating, bilinear) map
$\kappa_g:V\times V\rightarrow  k \oplus V$
also decomposes into constant and linear parts:
$$
\kappa_g = \kappa_g^C+\kappa_g^L
\quad\text{where}\quad
\kappa_g^C:V\wedge V\rightarrow k,\ \ \ 
\kappa_g^L:V\wedge V\rightarrow V\ .
$$
Now let 
$$
   \cH_{\kappa, t} := T(V)\# G [t]/(v w- w v 
                      -\kappa^L(v,w) t             
                     - \kappa^C(v,w)t^2
                            \mid v,w\in V).
$$
We call $\cH_{\kappa, t}$ a {\bf Drinfeld orbifold
algebra over $k[t]$} whenever $\cH_{\kappa}$ is a Drinfeld orbifold algebra;
in this case, $\cH_{\kappa,t}$
is a deformation of $S\# G$ over $k[t]$ and $\cH_{\kappa,t}/t\cH_{\kappa,t}
\cong S\# G$.

The following theorem extends~\cite[Theorem~3.2]{Witherspoon}
(in the case of a trivial twisting cocycle) to our setting.
We note that in case $\kappa^L \equiv 0$, a change of formal parameter
allows us to replace $t^2$ by $t$ in the definition of $\cH_{\kappa,t}$,
thus giving the Drinfeld Hecke algebras (i.e., graded Hecke algebras) 
over $k[t]$ (defined in~\cite{Witherspoon})
as a special case. We use standard notation for graded linear maps:
If $W$ and $W'$ are graded vector spaces, a linear map $\alpha: W \rightarrow W'$
is  homogeneous of {\bf degree} $\deg \alpha$ if 
$\alpha(W_i) \subseteq W_{i + \deg \alpha}$ for all $i$.

\begin{thm}
\label{degreemu_1}
The Drinfeld orbifold algebras $\cH_{\kappa,t}$ over $k[t]$ are precisely 
the deformations of $S\# G$ over $k[t]$ for which $\deg \mu_i = -i$ 
and for which $kG$ is in the kernel of $\mu_i$ for all $i\geq 1$.
\end{thm}

The hypothesis that $kG$ is in the kernel of all $\mu_i$ 
is a reasonable one when the characteristic of $k$ does not
divide the order of $G$: 
In this case  one may choose to work with maps 
that are linear over the semisimple 
ground ring $kG$ as in~\cite{BGS,EtingofGinzburg}. 
There are however alternative ways to express
Drinfeld Hecke algebras for which this hypothesis is not true.
See~\cite[Theorem~3.5]{RamShepler} for a comparison with
Lusztig's equivalent definition of a Drinfeld (graded) Hecke
algebra in which the group action relations are deformed.

\begin{proof}
Assume $\cH_{\kappa, t}$ is a Drinfeld orbifold algebra over $k[t]$.
Let $v_1,\ldots,v_n$ be a basis of the vector space $V$, so that
$$
  \{v_1^{i_1}\cdots v_n^{i_n}\mid i_1,\ldots,i_n\in {\Z_{\geq 0}} \}
$$
is a basis of $S$.
Since ${\rm gr} \cH_{\kappa}\cong S\# G$, there is a
corresponding basis $\cB$ of $\cH_{\kappa}$ 
given by all $v_1^{i_1}\cdots v_n^{i_n} g$,
where $g$ ranges over all elements in $G$ and $i_1,\ldots,i_n$ range over all
nonnegative integers.  Hence, we may identify $\cH_{\kappa,t}$
with $S\# G[t]$ as a $k$-vector space.  As $\cH_{\kappa,t}$
is associative, it defines a deformation of $S\#G[t]$ as follows.

Let $r=v_1^{i_1}\cdots v_n^{i_n}g$ and $s=v_1^{j_1}\cdots v_n^{j_n} h$ be
elements of $\cB$. 
For clarity, we denote the product in $\cH_{\kappa,t}$ by $*$.
Using the relations of $\cH_{\kappa, t}$ to express the product
$r*s$ as a linear combination of elements of $\cB$,
we may expand uniquely: 
$$
  r*s = rs +\mu_1(r,s)t + \mu_2(r,s)t^2+\cdots + \mu_m(r,s)t^m
$$
for some $m = m_{r,s}$ depending on $r,s$, and some $\mu_1,\ldots,\mu_m$.
By the definition of $\cH_{\kappa,t}$ as a quotient of $T(V)\# G[t]$, 
the group algebra $kG$ is in the kernel of $\mu_i$ for all $i$. 
Using the relations in $\cH_{\kappa}$, we have 
$$
  r*s = ((v_1^{i_1}\cdots v_n^{i_n})* ( {}^g (v_1^{j_1}\cdots v_n^{j_n})))\ gh.
$$
We apply the relations 
of $\cH_{\kappa,t}$ repeatedly to rewrite the product 
$(v_1^{i_1}\cdots v_n^{i_n})*( {}^g (v_1^{j_1}\cdots v_n^{j_n}))$
as an element in the $k$-span of $\cB$.
We prove by induction on the degree $d=\sum_{l=1}^n (i_l+j_l)$ that
$\deg \mu_i = -i$.
It suffices to prove this in case $g=1$.
If $d=0$ or $d=1$, the maps $\mu_i$ give 0, and so they satisfy the degree requirement
trivially. 
Similarly, whenever $a<b$,
$v_a*v_b$ in $\cH_{\kappa,t}$ identifies with $v_a v_b$ in $S$,
and $\mu_i(v_a, v_b)=0$ for all $i$. 
Thus 
if $d=2$, the nontrivial case is when some $i_l=1$ and some $j_m=1$
with $l>m$. Then 
$$
 v_{l}*v_{m}-v_{m}*v_{l} =  \kappa^L(v_{l}, v_{m}) t
   + \kappa^C(v_{l},v_{m}) t^2.
$$
By construction, $\mu_1(v_{l}, v_{m}) = \kappa^L(v_{l}, v_{m})$,
an element of $V\ot kG$, and
the map $\mu_1$ has degree $-1$ on this input.
Similarly, $\mu_2(v_{l}, v_{m}) = \kappa^C(v_{l},v_{m})$,
which has degree $-2$ on this input. 

Now assume $d>2$ is arbitrary and $d=\sum_{l=1}^n (i_l+j_l)$.
Without loss of generality, assume $i_n\geq 1$, $j_1\geq 1$,
and then
\begin{eqnarray*}
(v_1^{i_1}\cdots v_n^{i_n})*(v_1^{j_1}\cdots v_n^{j_n}) & = &
  (v_1^{i_1}\cdots v_n^{i_n-1})*(v_1v_n * v_1^{j_1-1}\cdots v_n^{j_n})\\
&&+ (v_1^{i_1}\cdots v_n^{i_n-1}) * (\kappa^L(v_n,v_1) * v_1^{j_1-1}
   \cdots v_n^{j_n}) t \\
  &&+ (v_1^{i_1}\cdots v_n^{i_n-1}) * (\kappa^C(v_n,v_1)
   * v_1^{j_1-1}\cdots v_n^{j_n}) t^2 .
\end{eqnarray*}
In the second and third terms, we see that the degree lost by applying the
map is precisely that gained in the power of $t$.
In the first term, no degree was lost and no power of $t$ was gained,
however the factors are one step closer to being part of a PBW basis.
By induction, the degrees of the $\mu_i$ are as claimed.
Equivalently, we may give $t$ a degree of 1, making $\cH_{\kappa,t}$
a graded algebra, and argue as in~\cite{BravermanGaitsgory,DCY}.

Now assume that $A$ is any deformation of $S\# G$ over $k[t]$
for which $\deg\mu_i=-i$ and for which $kG$ is in the kernel of
$\mu_i$ for all $i\geq 1$.
 By definition, $A$ is isomorphic to $S\# G [t]$ as a 
vector space over $k[t]$.
Fix a basis $v_1,\ldots,v_n$ of $V$.
Let $\phi: T(V)\# G[t]\rightarrow A$ be the $k[t]$-linear
map given by
$$
  \phi(v_{i_1}\cdots v_{i_m} g) = v_{i_1}*\cdots * v_{i_m} * g
$$
for all words $v_{i_1}\cdots v_{i_m}$ and group elements $g$.
Since $T(V)$ is free on $v_1,\ldots, v_n$ and by hypothesis,
$\mu_i (kG,kG)=\mu_i
(kG,V)=\mu_i(V,kG) = 0$ for all $i\geq 1$, the map $\phi$ is in
fact an algebra homomorphism.
It may be shown by induction on degree that $\phi$ is surjective,
using the degree hypothesis on the maps $\mu_i$.

We next find the kernel of $\phi$. Let $v,w\in V$ be elements of
the basis. Then
\begin{eqnarray*}
  \phi(vw) &=& v*w = vw + \mu_1(v,w)t + \mu_2(v,w)t^2\\
 \phi(wv) &= & w*v = wv + \mu_1(w,v)t +\mu_2(w,v)t^2
\end{eqnarray*}
since $\deg\mu_i = -i$ for each $i$.
Since $vw=wv$ in $S$, we have 
$$
  \phi(vw-wv) = (\mu_1(v,w)-\mu_1(w,v))t -
     (\mu_2(v,w)-\mu_2(w,v))t^2.
$$
It follows that
\begin{equation}\label{ideal}
vw-wv - (\mu_1(v,w)-\mu_1(w,v))t - (\mu_2(v,w)-\mu_2(w,v))t^2
\end{equation}
is in the kernel of $\phi$, since $\phi(vg)=vg$ and $\phi(g)=g$
for all $v\in V$ and $g\in G$.
By the degree conditions on the $\mu_i$, there are functions
$\kappa_g^L : V\wedge V \rightarrow V\otimes kG$ and
$\kappa^C_g: V\wedge V\rightarrow kG$ for all $g\in G$ such that
\begin{eqnarray}
  \mu_1(v,w)- \mu_1(w,v) & = & \sum_{g\in G} \kappa^L_g(v,w) g
   \label{eqn:kappa-mu} \\
  \mu_2(v,w)-\mu_2(w,v) & = & \sum_{g\in G} \kappa^C_g(v,w) g.
    \label{eqn:kappa-mu2}
\end{eqnarray}
For each $g\in G$, the functions 
$\kappa_g^C: V\wedge V \rightarrow kG$
and
$\kappa^L_g: V\wedge V \rightarrow V\otimes kG$ are linear
(by their definitions).
Let $I[t]$ be the ideal of $T(V)\# G [t]$ generated
by all expressions of the form (\ref{ideal}), so by definition
$I[t]\subset \ker \phi$.
We claim that in fact $I[t]=\ker \phi$: The quotient $T(V)\# G [t]/I[t]$
is by definition a filtered algebra over $k[t]$ whose associated graded 
algebra is necessarily $S\# G [t]$ or a quotient thereof.
By a dimension count in each degree, since $I[t]\subset \ker\phi$,
this forces $I[t]=\ker \phi$.
Therefore $\phi$ induces an isomorphism from $\cH_{\kappa,t}$ to $A$
and thus the deformation $A$ of $S\# G$ is isomorphic to a Drinfeld orbifold algebra.
\end{proof}

\begin{remark}\label{remark:kappa-mu}
{\em 
When working with a Drinfeld orbifold
algebra, we may always assume the relations (\ref{eqn:kappa-mu}) 
and (\ref{eqn:kappa-mu2}) hold for $v,w$ in $V$
as a consequence of the proof.
In a later section, we will make more explicit this connection between
the functions $\mu_i$ and $\kappa$, using Hochschild cohomology in
case the characteristic of $k$ does not divide the order of $G$: 
We will consider the $\mu_i$ to be cochains on the bar resolution of
$S\# G$, and $\kappa$ to be a cochain on the Koszul resolution of $S$.
The relations (\ref{eqn:kappa-mu}) and (\ref{eqn:kappa-mu2}) then result from
applying chain maps to convert between the two resolutions. 
Specifically, let 
$\phi_{\DOT}$ be a map from the Koszul resolution to the bar resolution
of $S$ (a subcomplex of the bar resolution of $S\# G$). 
Then  $\kappa^L = \mu_1 \circ\phi_2$, as we will explain.
}
\end{remark}

\section{Necessary and Sufficient Conditions}\label{sec:nands}
We determine conditions on the parameter
$\kappa$ for $\cH = \cH_\kappa$ to satisfy the PBW condition.
In the setting of  symplectic reflection algebras,
Etingof and Ginzburg~\cite[Theorem~1.3]{EtingofGinzburg}
used a generalization of results of Braverman and Gaitsgory
~\cite[Theorem~0.5 and Lemma~3.3]{BravermanGaitsgory}
that replaces the ground field $k$
with the (semisimple) group ring $kG$.
This approach was then adopted in
Halbout, Oudom, and Tang~\cite{HOT}.
Since one of the conditions in \cite{HOT} 
is missing a factor of 2, we include
two proofs of the PBW conditions for Drinfeld orbifold algebras, 
one using this generalization of work of Braverman and Gaitsgory,
and one using Bergman's Diamond Lemma~\cite{Bergman}.
The second proof applies in all characteristics other than 2,
even those dividing the order of $G$,
while the first requires $kG$ to be semisimple.
(See~\cite{Khare} for the Diamond Lemma argument applied 
in a related setting and see~\cite{LevandovskyyShepler} for a related
approach 
using noncommutative Gr\"obner theory, but
in a quantum setting.)

The set of all parameter functions
$$ \kappa: V\wedge V \rightarrow (k \oplus V)\otimes kG $$
(defining the quotient algebras $\cH_{\kappa}$)
carries the usual induced $G$-action:
$(^h\kappa)(*)=\,  
 ^h(\kappa(^{h^{-1}}(*)))$, i.e., for all $h\in G$ and $v,w\in V$, 
$$(^h\kappa)(v,w) = \ ^h \! ( \kappa ( {}^{h^{-1}} \! v, {}^{h^{-1}} \! w))
=\sum_{g\in G}\ ^h(\kappa_g(^{h^{-1}}v,^{h^{-1}}w))\, hgh^{-1}\ .$$
We say that $\kappa$ is $G$-invariant when 
$^h\kappa = \kappa$ for all $h$ in $G$.
Let $\Alt_3$ denote the cyclic group of order 3 considered as
a subgroup of the symmetric group on 3 symbols. 
Note that the following theorem gives conditions in the symmetric algebra $S$.
We allow the field $k$ to have arbitrary characteristic other than 2. 

\begin{thm}\label{thm:4condns}
The algebra 
$\cH_{\kappa}$ is a Drinfeld orbifold algebra if and only if the following
conditions hold 
for each
$g$ in $G$ and $v_1,v_2,v_3$ in $V$:
\begin{itemize}
\item[(i)] The parameter function $\kappa$ is $G$-invariant,
\item[(ii)] 
\rule[0ex]{0ex}{4ex} 
$\displaystyle{\sum_{\sigma\in\Alt_3}}\, \kappa^L_g (v_{\sigma(2)}, v_{\sigma(3)})
   (v_{\sigma(1)}- {}^g v_{\sigma(1)}) = 0$ in $S=S(V)$,
\item[(iii)]
\rule[0ex]{0ex}{4ex} 
\vspace{-4ex}
$$
\begin{aligned}
\hphantom{xx}
\sum_{\sigma\in \Alt_3} \sum_{h\in G} \, 
\kappa^L_{gh^{-1}}
    \big(
 v_{\sigma(1)}+\ ^hv_{\sigma(1)},\ & \kappa^L_h(
v_{\sigma(2)}, v_{\sigma(3)})\big)\\
   &=2 \sum_{\sigma\in \Alt_3}\ \kappa_g^C(v_{\sigma(2)}, v_{\sigma(3)})
    ( {}^g v_{\sigma(1)} - v_{\sigma(1)})\ ,
\end{aligned}
$$
\item[(iv)] 
\rule[0ex]{0ex}{4ex} 
$\displaystyle{\sum_{\sigma\in\Alt_3}\sum_{h\in G}}\ 
\kappa^C_{gh^{-1}}\big(
v_{\sigma(1)}+\ ^hv_{\sigma(1)},
\kappa^L_h(v_{\sigma(2)}, v_{\sigma(3)})\big) = 0$.
\end{itemize}
\end{thm} 
\begin{proof}[Proof of Theorem~\ref{thm:4condns} using the theory
of Koszul rings over $kG$]
In this proof, we restrict to the case where the 
characteristic of $k$  does not divide the order of $G$.
The skew group algebra $S\# G$ is then a Koszul ring over $kG$,
as defined by Beilinson, Ginzburg, and Soergel 
(see \cite[Definition 1.1.2 and Section 2.6]{BGS}).
These authors worked with graded algebras in which the degree 0 component
is not necessarily commutative, but is a semisimple algebra.
In our case the degree 0 component of $S\# G$ is the semisimple
group algebra $kG$.
The results of Braverman and Gaitsgory \cite[Theorem~0.5 and Lemma~3.3]{BravermanGaitsgory}
can be extended to 
this general setting to give necessary and sufficient conditions
on $\kappa$ under which $\cH_{\kappa}$ is a Drinfeld orbifold algebra
(cf.~\cite[Proof of Theorem 1.3]{EtingofGinzburg}).
We give $V\ot kG$ a $kG$-bimodule structure as follows:
$ g (v\ot h) := {}^g v\ot gh$ and $(v\ot h)g := v\ot hg$ for
all $v\in V$, $g,h\in G$.
Note that $T:= T_{kG}(V\ot kG)$ is isomorphic to $T(V)\# G$ as 
an algebra. 

Let $P$ be the sub-$kG$-bimodule of $T$
generated by all $v\ot w - w\ot v -\kappa(v,w)$, for $v,w\in V$.
Let $R$ be the sub-$kG$-bimodule generated by all $v\ot w - w\ot v$,
for $v,w\in V$. Let $F$ be the standard filtration on $T$, that is, 
$F^0(T)=kG$, $F^1(T) = kG
\oplus (V\ot kG)$, $F^2(T) = kG\oplus (V\ot kG) \oplus (V\ot V\ot kG)$, 
and so on.
By \cite[Theorem~0.5]{BravermanGaitsgory}, $\cH_{\kappa}
\cong T_{kG}(V\ot kG)/P$ is a Drinfeld orbifold algebra if
and only if 
\begin{itemize} 
 \item[(I)] $\ P\cap F^1(T) = 0$ and 
\item[(J)] $\ (F^1(T)\cdot P \cdot F^1(T)) \cap F^2(T) = P$.
\end{itemize}
By \cite[Lemma~3.3]{BravermanGaitsgory}, if (I) holds, then (J) is equivalent to
the following three conditions, where $\alpha: R\rightarrow V\ot kG$,
$\beta: R\rightarrow kG$ are maps for which 
$$
P=\{r - \alpha(r)-\beta(r)\mid r\in R\}:
$$
\begin{itemize}
\item[(a)] $\Ima (\alpha\ot \id - \id\ot \alpha) \subset R$.
\item[(b)] $\alpha\circ (\alpha\ot \id - \id\ot\alpha) = - (\beta\ot\id 
-\id\ot\beta)$
\item[(c)] $\beta\circ (\id\ot\alpha - \alpha\ot\id) \equiv 0$.
\end{itemize}
The above maps $\alpha\ot\id - \id\ot\alpha$ and 
$\beta\ot\id - \id\ot\beta$ are defined on the intersection 
$(R\ot_{kG} (V\ot kG))\cap ((V\ot kG)\ot_{kG} R)$. 
Extend $\kappa$ to an alternating $k G$-module map
on $T^2:= T_{kG}^2(V\ot k G)$, so that
$\kappa(g(v\ot w)h)=\kappa(\, ^gv\ot\ ^g w)\,gh$ for all $g,h$ in $G$
and $v,w$ in $V$.
Then
$\alpha (v \ot w - w \ot v) = \kappa^L(v \ot w)=
(1/2) \kappa^L (v \ot w - w \ot v)$ (as $\kappa$ is alternating) and 
for all $r$ in $R$,
$$2\alpha(r)= \kappa^L(r)\ .$$
(Similarly, $2\beta(r)= \kappa^C(r)$ for all $r$ in $R$.)

First note that (I) is equivalent to the condition that $G$ preserves
the vector space generated by all $v\ot w - w\ot v - \kappa^L(v,w)
-\kappa^C(v,w)$, i.e., this space contains
$$
  {}^g v \ot {}^g w - {}^g w\ot {}^g v - {}^g(\kappa^L(v,w))
   - {}^g (\kappa^C(v,w))
$$
for each $g\in G$, $v,w\in V$.
Equivalently, $\kappa^L( {}^gv, {}^g w) =
{}^g(\kappa^L(v,w))$ and $\kappa^C({}^gv,{}^gw) = {}^g (\kappa^C(v,w))$, i.e., 
both $\kappa^L$ and $\kappa^C$ are $G$-invariant, yielding
Condition~(i) of the theorem.

We assume now that $\kappa$ is $G$-invariant and proceed with the 
remaining conditions.

Condition~(a): 
As a $kG$-bimodule, $(R\ot_{kG} (V\ot kG))\cap((V\ot kG)\ot_{kG} R)$ is
generated by elements of the form $\sum_{\sigma\in S_3} (\sgn\sigma )
v_{\sigma(1)}\ot v_{\sigma(2)}\ot v_{\sigma(3)}$, so we find the image
of $\alpha\ot\id - \id\ot\alpha$ on these elements.
After reindexing, we obtain
$$
  \sum_{\sigma\in\Alt_3} (\kappa^L(v_{\sigma(2)},v_{\sigma(3)})\ot v_{\sigma(1)}
  - v_{\sigma(1)}\ot \kappa^L(v_{\sigma(2)}, v_{\sigma(3)})).
$$
We decompose into components indexed by $g$ in $G$ and 
shift all group elements to the right (tensor products are over $kG$).
The $g$-th summand is then
\begin{equation}\label{elt-R}
  \sum_{\sigma\in \Alt_3} (\kappa^L_g(v_{\sigma(2)},v_{\sigma(3)})\ot
   {}^g v_{\sigma(1)} - v_{\sigma(1)}\ot \kappa^L_g(v_{\sigma(2)},
   v_{\sigma(3)})) g,
 \end{equation}
which must be an element of $R$. This is equivalent to the
vanishing of its image in $S\# G$.
We rewrite this as Condition~(ii) of the theorem. 

Condition~(b): 
We assume Condition~(a) holds and thus~(\ref{elt-R}) is an element of $R$.
We compute the left side of Condition~(b) by applying $\alpha$ to this element.
Since $2\alpha(r)=\kappa^L(r)$ for all $r$ in $R$,
we obtain the left side of Condition~(iii) of the theorem
after dividing by 2.
Similarly, it is not difficult to see that
the right side of Condition~(b) agrees with the 
right side of Condition~(iii) of the theorem:
The image of $\sum_{\sigma\in S_3}(\sgn\sigma) v_{\sigma(1)}
\ot v_{\sigma(2)}\ot v_{\sigma(3)} $ under $- (\beta\ot\id - \id\ot\beta)$ is
$$
  -\sum_{\sigma\in\Alt_3, \ g\in G} (\kappa^C_g(v_{\sigma(2)},v_{\sigma(3)})
   \ot {}^g v_{\sigma(1)} - v_{\sigma(1)}\ot \kappa^C_g (v_{\sigma(2)}, v_{\sigma(3)})) g 
$$
(as an element of $V\ot kG$) which we rewrite as
$$
  - \sum_{\sigma\in\Alt_3, \ g\in G} (\kappa^C_g(v_{\sigma(2)}, v_{\sigma(3)})
  ( {}^g v_{\sigma(1)} - v_{\sigma(1)})) g.
$$

Condition~(c): An analysis similar to that for Condition~(b) yields 
Condition~(iv) of the theorem. 
\end{proof}

\begin{proof}[Proof of Theorem~\ref{thm:4condns} using the Diamond Lemma]
In this proof, the characteristic of $k$ may be 0 or any odd prime. 
We apply~\cite{Bergman} to obtain 
conditions on $\kappa$ equivalent to existence of a 
PBW basis and then argue that these conditions
are equivalent to those in the theorem.
We suppress details and merely record
highlights of the argument (which requires one to 
fix a monomial ordering and check
all overlap/inclusion ambiguities on
the set of relations defining $\cH_{\kappa}$),
as described, for example, in~\cite[Chapter~3]{AlgorithmicMethods}).
Fix a basis $v_1,\ldots, v_n$ of $V$ and
let $\cB$ be our prospective PBW basis: Set 
$
\cB=\{v_1^{\alpha_1}\cdots v_n^{\alpha_n}g:
\alpha_i\in \ZZ_{\geq 0}, g\in G\}
\subset T(V)\ot kG\ ,
$ 
a subset of the free algebra $\F$ generated by
$v$ in $V$ and $g$ in $G$. 

Using the Diamond Lemma, one may show that necessary and sufficient conditions 
for $\cH_{\kappa}$ to satisfy the PBW condition arise from expanding 
conjugation and Jacobi identities in $\cH_{\kappa}$:
For every choice of parameter $\kappa$, and for
every $v,w$ in $V$ and $h$ in $G$, the elements
$$
\begin{aligned}
& \quad (1)\ \ \ h[v,w]_{\cH}\, h^{-1} - [\,^hv,\, ^hw]_{\cH}\, , \quad\text{and}\\
& \quad (2)\ \ \ [v_i,[v_j,v_k]_{\cH}]_{\cH}\, +[v_j,[v_k,v_i]_{\cH}]_{\cH}\,
 +[v_k,[v_i,v_j]_{\cH}]_{\cH}\ 
\end{aligned}\\
$$ 
are always zero in the associative algebra $\cH$.
Here, $[a,b]_{\cH}:=a b-b a$ is just the commutator in $\cH$ of $a,b\in \cH$.
Using the relations defining $\cH$, we move all group elements to the 
right and arrange indices of basis vectors in increasing order
(apply straightening operations).

An analysis of elements of type (1) shows that 
a PBW property on $\cH_{\kappa}$ forces 
$^h\kappa = \kappa$ for all $h$ in $G$.  Indeed, this condition
is equivalent to 
$$
\kappa_{h^{-1}gh}(v,w)\ =\ \ ^{h^{-1}}(\kappa_g(\,^hv,\, ^hw))
\quad\text{for all } g,h\in G,\ v,w \in V\ .
$$

We next write each element of type (2) above in the image 
under the projection map
$\pi:\cF \rightarrow \cH$ 
of some $f(v_i, v_j, v_k)$ 
in the $k$-span of (potentially) nonzero elements of $\cB$.
In $\cH_\kappa$,
$$
\begin{aligned}
\
[v_1,&[v_2,v_3]_{\cH}]_{\cH}
= v_1 \kappa(v_2, v_3) - \kappa(v_2, v_3) v_1\\
%
&=
\sum_{g\in G}
v_1 \kappa_g^C(v_2,v_3)g - \kappa_g^C(v_2, v_3)g v_1
+v_1 \kappa_g^L(v_2,v_3)g - \kappa_g^L(v_2, v_3)g v_1\\
&=
\sum_{g\in G}
\Big(v_1 \kappa_g^C(v_2,v_3) - \kappa_g^C(v_2, v_3)\, ^g v_1
+v_1 \kappa_g^L(v_2,v_3) - \kappa_g^L(v_2, v_3)
\,^g v_1\Big)g\\
&=
\sum_{g\in G}
\Big(\kappa_g^C(v_2,v_3)(v_1 -\,  ^g v_1)
+v_1 \kappa_g^L(v_2,v_3) - \kappa_g^L(v_2, v_3)
\,^g v_1\Big)g\ .\\
\end{aligned}
$$
We apply further relations in $\cH$ 
to this last expression to rearrange the vectors $v_1, \ldots, v_n$ 
by adding terms of lower degree.  Thus, if we express $f$ as $f_0+f_1+f_2$ 
where $f_i$ has degree $i$ in the
free algebra $\cF$, 
then
$\pi(f_2)$ and 
$$
\sum_{g\in G}
\sum_{\sigma\in \Alt_3}
\kappa_g^L(v_{\sigma(2)},v_{\sigma(3)})
(v_{\sigma(1)}-\, ^g v_{\sigma(1)})\ g\ \\
$$
differ only by a rearrangement of vectors: 
They both project to the same element
under $T(V)\otimes kG \rightarrow S(V)\otimes kG$.
But $f_2$ is zero in the free algebra $\cF$ 
if and only if its image is zero in $S(V)\otimes kG$,
yielding Condition~(ii) of the theorem.

The other conditions of the theorem require a bit of manipulation.
One may show that
$$\begin{aligned}
f&=
\sum_{\sigma\in{\text{Alt}}_3 \atop g \in G}
\Big[
\kappa_g^C(v_{\sigma(2)},v_{\sigma(3)})
(v_{\sigma(1)} -\ ^g v_{\sigma(1)})
+
\sum_{a<b}
(D^g_{ab}+D^g_{ba})
\ v_a v_b
-
D^g_{ab}\ \kappa(v_a,v_b)\ 
\Big]g\ 
\end{aligned}
$$
where the $D^g_{ab}$ in $k$ are constants
determined by the action of $G$ on $V$
and the values of $\kappa$
expanded in terms of the fixed basis of $V$.
Specifically, 
$
D^g_{ab}\ :=\ 
\delta_{b,\sigma(1)} c_a^{\sigma(2),\sigma(3),g}
-c_b^{\sigma(2),\sigma(3),g}\, d_a^{\,\sigma(1), g}
$
where $\delta_{a,b}$ is the Kronecker delta symbol 
and where
$
^gv_a=\sum_{b} d_b^{\,a,g} v_b$
and $\kappa_g^L(v_a, v_b)
=\sum_m c_m^{a,b,g} v_m$ .

Note that $f_2$ is zero if and only if
$\sum_{\sigma\in{\text{Alt}}_3} (D_{ab}^g+ D_{ba}^g)=0$
for all $a<b$ and $g$ in $G$.
Thus, whenever $f_2$ is zero,
we may substitute $D_{ab}^g = -D_{ba}^g$ 
in the equation $0=2(f_0+f_1)$ to see that $f_0+f_1$
vanishes exactly when
$$
2 \sum_{g \in G}\ 
\kappa^C(v_{\sigma(2)},v_{\sigma(3)})
(v_{\sigma(1)} -\ ^g v_{\sigma(1)}) g
=
\sum_{\sigma\in{\text{Alt}}_3\atop g \in G}\
\sum_{a<b}\
(D^g_{ab}-D^g_{ba})
\kappa(v_a, v_b) \ g\ .
$$
We write the right-hand side as a sum over
all $a$ and $b$ (as $\kappa$ is alternating)
and obtain
$$\sum_{\sigma\in{\text{Alt}}_3 \atop g \in G}\
\kappa\Big(
\kappa^L_g
(v_{\sigma(2)},v_{\sigma(3)}),
v_{\sigma(1)}+\ ^gv_{\sigma(1)}
\Big)\ g\ .
$$
This yields Conditions~(iii) and~(iv) of the theorem
whenever Condition~(ii) holds.

Thus, the four conditions of the theorem are equivalent
to $G$-invariance of $\kappa$ and the vanishing of
all $f_0, f_1, f_2$ (for any $i,j,k$),
which in turn is equivalent to 
the PBW property for $\cH_{\kappa}$ by careful application of the Diamond Lemma.
\end{proof}

We illustrate the theorem by giving two examples for which
$\kappa^C$ is identically 0. In the next section we give
an example for which $\kappa^L$ and $\kappa^C$ are both nonzero.

\begin{example}\label{example:klein4group}
{\em
Let $G\cong \Z/2\Z\times \Z/2\Z$, with generators $g$ and $h$,
act on the complex vector space $V$ having basis $x,y,z$ by:
$$
\begin{aligned} & {}^gx = -x, \ \ \  {}^gy=y, \ \ \  {}^gz= -z,\\
  & {}^hx = -x, \ \ \  {}^hy = -y, \ \ \  {}^h z = z.
\end{aligned}
$$
Define an alternating bilinear map 
$\kappa^L: V\times V\rightarrow V\ot kG$ by
$$
   \kappa^L(x,y) = zh, \ \ \ \kappa^L(y,z) = xgh, \ \ \ \kappa^L(z,x)=yg,
$$
and let $\kappa^C \equiv 0$. One may check that $\kappa^L$ is 
$G$-invariant and that Conditions~(ii) and~(iii) of
Theorem~\ref{thm:4condns} hold. Condition~(iv) holds automatically since
$\kappa^C$ is identically 0.
The corresponding Drinfeld orbifold algebra is
$$
   T(V)\# G / ( \ [x,y]-zh, \ [y,z] - xgh, \ [z,x] -yg \ ).
$$
}\end{example}

\quad

\begin{example}\label{example:S3}
{\em
Let $G=S_3$ act by permutations on a basis $v_1,v_2,v_3$
of a complex three-dimensional vector space $V$.
Let $\xi$ be a primitive cube root of 1, and let
$$
    w_1 = v_1+\xi v_2+\xi^2 v_3, \ \ \ w_2=v_1 + \xi^2 v_2 + \xi v_3, \ \ \  
    w_3 = v_1 + v_2+ v_3.
$$
Define an alternating bilinear map 
$\kappa^L : V\times V \rightarrow V\ot kG$ by
$$
   \kappa^L(w_1,w_2) = w_3 ((1,2,3) - (1,3,2)),  \ \ \
    \kappa^L(w_2,w_3) =0, \ \ \ \kappa^L(w_1,w_3) = 0,
$$
and let $\kappa^C \equiv 0$, where $(1,2,3), (1,3,2)$ are the standard 
3-cycles in $S_3$. One may check that $\kappa^L$ is  $G$-invariant
and that Conditions~(ii) and~(iii) of Theorem~\ref{thm:4condns} hold. 
Condition~(iv) holds automatically since $\kappa^C$ is identically zero.
The corresponding Drinfeld orbifold algebra is 
$$
  T(V)\# G/ ( \ [w_1,w_2] - w_3((1,2,3) - (1,3,2)) , \ \ \ 
     [w_2,w_3] , \ \ \ [w_1, w_3] \ ).
$$  
}\end{example}

\quad

The conditions of Theorem~\ref{thm:4condns}
simplify significantly when $\kappa^L$ is supported on 
the identity element 
$1:=1_G$ of $G$ alone, 
and we turn to this interesting case in the next section.

\section{Lie Orbifold Algebras}\label{sec:loa}

The universal enveloping algebra of a finite-dimensional Lie algebra
is a special case of a Drinfeld orbifold algebra.
We extend universal enveloping algebras 
by groups and explore deformations of the resulting algebras
in this section.
Assume throughout this section that the linear part of our parameter $\kappa$ is supported on the identity $1=1_G$ of $G$ alone,
that is, $\kappa^L_g \equiv 0$ for all $g\in G - \{1\}$.
It is convenient in this section to use standard notation from 
the theory of Lie algebras
and Drinfeld Hecke algebras
(i.e., graded Hecke algebras): 
Let $$a_g: V\wedge V \rightarrow k \ \ \ \ \ (\mbox{for all }g\mbox{ in } G)$$
and
$$[\cdot, \cdot]_{\g}:V\wedge V \rightarrow V$$
be linear functions where $\g := V$ as a vector space with the additional
structure given by the map $[\cdot , \cdot ]_{\g}$. 
Define an algebra $\cH:=\cH(\g; a_g, g\in G)$ as the quotient
$$\cH = T(V) \# G   /    
(   
v  w - w  v  - [v,w]_{\g}-\sum_{g\in G} a_g(v,w) g
\mid v,w\in V ) .
$$
Then $\cH$ is a filtered algebra by its definition.
We say that $\cH$ is a {\bf Lie orbifold algebra} when it satisfies the
PBW condition, that is, when gr$\cH \cong S\# G$.
We determine necessary and sufficient
conditions on the functions $[\cdot, \cdot]_{\g}$ and $a_g$ for $\cH$ to be  
a Lie orbifold algebra.
We will see that the PBW condition implies that $[\cdot,\cdot]_\g$
defines a Lie bracket on $V$, thus explaining the choice of notation 
and terminology.

We first examine the Jacobi condition of 
Theorem~\ref{thm:4condns}
resulting from the Jacobi identity on $\cH$,
$$
0=[v_1,[v_2,v_3]_{\cH}]_{\cH}+[v_2,[v_3,v_1]_{\cH}]_{\cH}+[v_3,[v_1,v_2]_{\cH}]_{\cH}\\
\quad\text{for all}\ v_i\in V\ ,
$$
after taking
$\kappa^L (v,w) = [v,w]_{\mathfrak g}$ and
$\kappa^C(v,w) = \sum_{g\in G} a_g(v,w) g$:
\begin{lemma}\label{lemma:J}
Conditions~(ii), (iii), and (iv)
of Theorem~\ref{thm:4condns}
(i.e., the Jacobi condition)
hold  for
$$\cH = T(V) \# G   /    
(  
v \otimes w - w \otimes v  - [v,w]_{\g}-\sum_{g\in G} a_g(v,w,) g
\mid v,w,\in V )
$$
if and only if three conditions are met:
$$
\begin{aligned}
& 1.\ \ 
[\cdot , \cdot]_{\g}  \ \text{endows}\  \g:=V \  \text{with the structure
of a Lie algebra.} \\
& 2. \ \ 
\text{The Jacobi identity for Drinfeld Hecke algebras holds:}\ \ 
\text{for all } v_i\text{ in } V,\, g \text{ in } G,\\ 
& \quad\quad\quad 0 
 =\    a_g(v_2,v_3) (v_1 -\, ^g v_1) 
+
a_g(v_3,v_1) (v_2 -\, ^g v_2) 
+
a_g(v_1,v_2) (v_3 -\, ^g v_3)\ .\\
& 3. \ \ 
\text{The constant part of the parameter is compatible with the Lie bracket:}\\
& \quad\quad
0 =\ 
a_g(v_3, [v_1,v_2]_\g) 
+a_g(v_1,[v_2,v_3]_\g) 
+a_g(v_2, [v_3,v_1]_\g) \ \ \text{for all } 
v_i\text{ in } V,\, g \text{ in } G\, .
\end{aligned}
$$
\end{lemma}

\begin{proof}
When $g=1$, Theorem~\ref{thm:4condns} (ii) holds automatically, and (iii) is 
equivalent to the Jacobi identity on $[\cdot, \cdot ]_{\g}$.
For $g\neq 1$, Theorem~\ref{thm:4condns} (iii) becomes 
$$
  0 = \sum_{\sigma\in\Alt_3} a_g (v_{\sigma(2)}, v_{\sigma(3)}) ( {}^g v_{\sigma(1)} - v_{\sigma(1)}),
$$
which is Condition~(2) in the statement of the lemma.
Finally, Theorem~\ref{thm:4condns}~(iv) becomes Condition~(3) in this lemma.
\end{proof}

Thus
when $\cH$ satisfies the PBW condition, we may view the vector space $V$  as a Lie algebra $\g$
under a Lie bracket $[\cdot, \cdot]_{\g}$.
We now are ready to give conditions for $\cH$ to be a Lie orbifold 
algebra in terms of Drinfeld Hecke algebras.
The conditions on the bracket  $[\cdot, \cdot]_{\g}$ stated in the
following proposition are immediate consequences of 
Theorem~\ref{thm:4condns}(i) and Lemma~\ref{lemma:J}(1).  We use
the notion of ``compatible'' structures defining
a Lie bracket
and Drinfeld Hecke algebra as defined
in Lemma~\ref{lemma:J}(3).  
\begin{prop}
\label{PBW}
The quotient
$$\cH = T(V) \# G   /    
(
v \otimes w - w \otimes v  - [v,w]_{\g}-\sum_{g\in G} a_g(v,w,) g
\mid v,w,\in V )
$$
defines a Lie orbifold algebra
if and only if
$$
\begin{aligned}
&\ \ \ 1.\ 
\text{The bracket } [\cdot , \cdot]_{\g} \ \text{endows}\ \g:=V\   
\text{with the structure} \text{ of a Lie algebra } \text{upon }\\
& \quad\quad\ \ 
\text{which }G \text{ acts as automorphisms,} \\
&\ \ \ 2. \ \ 
\text{The parameters $\{a_g\}_{g\in G}$ define a Drinfeld Hecke algebra, and}\\
&\ \ \ 3. \ \
\text{The Lie bracket structure is compatible 
with the Drinfeld Hecke algebra }\\
& \quad\quad\ \
\text{structure}.
\end{aligned}
$$
\end{prop}

In the nonmodular setting, 
we may use previous analysis of Drinfeld Hecke algebras
(i.e., graded Hecke algebras)
to state the last proposition in more detail. 
The conditions on the functions $a_g$ in
the next proposition result from a comparison of 
Lemma~\ref{lemma:J}(2) and the invariance condition of 
Theorem~\ref{thm:4condns}(i)
with~\cite[Lemma~1.5, equations~(1.6) and~(1.7), 
and Theorem~1.9]{RamShepler}. 
\begin{prop}\label{overC}
Suppose that $\text{char}(k)$ does not divide $|G|$.
Then
$\cH_{}$  is a Lie orbifold algebra 
if and only if
$$
\begin{aligned}
& \ \ 1.\ \ \ \
\text{The map} \ [\cdot , \cdot]_{\g} \ \text{is a $G$-invariant Lie bracket,}\\
& \ \ 2a. \ \ \ 
\text{The parameters $\{a_g\}_{g\in G}$ are determined
on conjugacy classes, with} \hspace{1cm}\\
& \quad\quad\quad\quad
a_{hgh^{-1}}(v,w)=\ a_{g}(\, ^hv,\, ^hw)
\quad\text{for all } v,w \in V,\   g,h\in G \ ,\\
& \ \ 2b. \ \ \ 
\text{For each $g\neq 1$, either $a_g\equiv 0$ or}
\ \ker a_g = V^g \ \text{with}
\ \codim (V^g)=2, \\
& \ \ 3. \ \ \ \
\text{Each $a_g$ is compatible with the Lie bracket: for all } v_i\in V
\mbox{ and } g\in G\, ,\\
&\quad\quad\quad\quad
0\  = \ 
a_g(v_3, [v_1,v_2]_\g) 
+a_g(v_1,[v_2,v_3]_\g) 
+a_g(v_2, [v_3,v_1]_\g) \ .
\end{aligned}
$$
\end{prop}

We further interpret the restrictive 
Condition~(3) in Proposition~\ref{overC}: 
Fix $g$ in $G$ with $\codim V^g=2$. Choose vectors 
$v_1, v_2$ spanning $(V^g)^\perp$ and $v_3, \ldots, v_n$ spanning $V^g$.
This condition then tells us that after expanding with respect to 
the basis $v_1, \ldots, v_n$, the
coefficient of $v_2$ in $[v_2, v_i]_\g$ is equal to 
the coefficient of $v_1$ in $[v_i, v_1]_\g$
for all $i\geq 3$. 

\begin{remark} 
{\em We view Lie orbifold algebras as generalizations both of 
symplectic reflection algebras and of universal enveloping algebras
with group actions. Indeed,
when $\cH_{}$ is a Lie orbifold algebra, 
we can replace each function $a_g$ with the zero function and 
recover the skew group algebra $\mathcal{U}\# G$
for $G$ acting as automorphisms on
$\mathcal{U}$, the universal enveloping algebra of a 
finite dimensional Lie algebra.
Alternatively,
we can replace the linear parameter by zero,
i.e., replace $[\cdot, \cdot]_{\g}$ by the zero bracket, 
and recover a Drinfeld Hecke algebra (a symplectic reflection algebra
in the special case that $G$ acts symplectically).
Thus, Lie orbifold algebras also include Drinfeld Hecke algebras 
(and Lusztig's graded affine Hecke algebra, in particular)
as special cases.
We illustrate by giving details for the example mentioned in
the introduction. 
}
\end{remark}

\begin{example}\label{example:sl2}
{\em Let $\mathfrak{g} = \mathfrak{sl}_2$ over $\CC$ with
basis $e$, $f$, $h$ and Lie bracket defined by 
$$
   [e,f]=h, \ \ \ [h,e]=2e, \ \ \ [h,f]=-2f.
$$
Let $G$ be a cyclic group of order 2 generated by $g$ acting on
$\mathfrak{sl}_2$ by 
$$
   {}^ge=f, \ \ \ {}^gf=e, \ \ \ {}^gh=-h.
$$
The bracket is $G$-invariant under this action.
Let  $a_g$ be the skew-symmetric form
on $V={\mathfrak{sl}}_2$ defined by
$$
  a_g(e,h)=1, \ \ \ a_g(h,f)=1, \ \ \ a_g(f,e)=0.
$$
This function is $G$-invariant, and $\ker a_g=V^g$ is the linear
span of $e+f$, which has codimension 2 in $V$. Furthermore, $a_g$ is compatible
with the Lie bracket, that is, 
Condition~(3) of Lemma~\ref{lemma:J} holds.
(It suffices to check this condition for 
$v_1=e$, $v_2=f$, $v_3=h$:  
$$
\begin{aligned}
a_g(h, [e,f]) + a_g (e, [f,h]) + a_g (f, [h,e])
 & = 
0\ .)
\end{aligned}
$$
Set $a_1$ equal to the zero function.
Then
$T(V)\# G$ modulo the ideal generated by
$$
 eh-he +2e - g, \ \ 
 hf-fh +2f -g, \ \ 
 ef-fe -h 
$$
is a Lie orbifold algebra.

In fact, Theorem~\ref{thm:4condns}  shows there are only
two parameters' worth of 
Lie orbifold algebras capturing this action of $G$
on ${\mathfrak{sl}}_2$:
Every such Lie orbifold algebra has the form
$$
T(V)\# G
/
(
 eh-he +2e- t_2g+t_1, \
 hf-fh +2f -t_2g+t_1, \  
 ef-fe -h  )
 $$
for some scalars $t_1,t_2$ in $\CC$.
(Note that $t_1=t_2=0$ defines the universal
enveloping algebra extended by $G$.)
}
\end{example}
\section{Koszul Resolution}\label{sec:koszul}
Hochschild cohomology catalogues and illuminates deformations of an 
algebra.  Indeed, every deformation of a $k$-algebra $R$ 
corresponds to an element in degree~2 Hochschild cohomology, $\HH^2(R)$.
Isomorphic deformations define cohomologous
cocycles.  We isolated in~\cite{SheplerWitherspoon1} 
the cocycles that define Drinfeld
Hecke algebras (i.e., graded Hecke algebras): 
Drinfeld Hecke algebras are precisely
those deformations of $S\# G$ whose 
corresponding Hochschild 2-cocycles are ``constant''.  
In the next section, we explain this statement, and we 
more generally express conditions for 
a quotient of $T(V)\# G$
to define a Drinfeld orbifold algebra in terms of 
the Hochschild cohomology of $S\# G$ and its graded Lie structure.
Here we establish preliminaries and notation.
{\em From now on, we assume that the characteristic of $k$
does not divide the order of $G$ and that $k$ contains 
the eigenvalues
of the actions of elements of $G$ on $V$.}
(For example, take $k$ algebraically closed of characteristic coprime
to $|G|$.)

Recall that we denote the image of $v$ in $V$ under the action of 
any $g$ in $G$ by ${}^gv$.  
Write $V^*$ for the contragredient (or dual) representation.
Given any basis $v_1,\ldots,v_n$ of $V$, let
$v_1^*,\ldots,v_n^*$ denote the dual basis of $V^*$.
Given any set $A$ carrying an action of $G$,
we write $A^G$ for the subset of elements invariant under the action.
Again, we write $V^g$ for the $g$-invariant subspace of $V$.
Since $G$ is finite, we may assume $G$ acts by isometries on $V$ (i.e.,
$G$ preserves a Hermitian form on $V$).

The Hochschild cohomology 
$\HHD(S\#G)$ is the space $\Ext^{\DOT}_{(S\# G)^e}(S\# G, S\# G)$, 
where $(S\# G)^e = (S\# G) \ot (S\# G)^{op}$ acts on $S\# G$ by multiplication,
one tensor factor acting on the left and the other 
tensor factor acting on the right. 
We also examine the Hochschild cohomology $\HH^{\DOT}(S, S\# G):= \Ext^{\DOT}_{S^e}
(S, S\# G)$ where $S^e = S\ot S^{op}$ 
and, more generally, $\HH^{\DOT}(S,M):= \Ext^{\DOT}_{S^e}(S,M)$
for any $S^e$-module $M$.

Let $\mathcal C$ be  a set of representatives of the conjugacy classes of $G$. 
For any $g$ in $G$, let $Z(g)$ be the centralizer of $g$.
Since we have assumed that 
the characteristic of $k$ does not divide the order of $G$,
there is a $G$-action giving the first  of the following
isomorphisms of graded vector spaces (see, for example, 
\c{S}tefan~\cite[Cor.~3.4]{Stefan}):
\begin{equation}\label{isos}
\begin{aligned}
  \HHD(S\# G) \ \ 
&\cong \ \ 
\HHD(S, S\#G)^G\\
\ \ &\cong \ \ 
 \Biggl(\,\bigoplus_{g\in G} \HHD(S, S g)
\Biggr)^{G}  
\ \ \cong \ \ 
\bigoplus_{g\in{\mathcal C}}\HHD(S,S g)^{Z(g)}.
\end{aligned}
\end{equation}
The first line is in fact a graded algebra isomorphism; 
it follows from applying a spectral sequence.
The second isomorphism results from decomposing 
the bimodule $S\# G$ into the direct sum of components $S\, g$.
The action of $G$ permutes these components via the conjugation action of $G$
on itself, and thus the third isomorphism is a canonical projection onto a set of 
representative summands.  
Each space $\HHD(S,S g) = \text{Ext}^{\DOT}_{S^e}(S, S g)$ may be determined
explicitly using the Koszul resolution of $S$
(a free resolution of $S$ as an $S^e$-module) that we recall next.

The {\bf Koszul resolution} $K_{\DOT}(S)$ is defined by 
$K_0(S)=S^e$ and
\begin{equation}\label{koszul-res2}
  K_{p}(S) = S^e \ot \Wedge^{p}(V)
\end{equation}
for $p\geq 1$, with differentials
\begin{equation}\label{koszul-diff}
d_p(1\ot 1\ot v_{j_1}\wedge\cdots\wedge v_{j_p}) = 
   \sum_{i=1}^p (-1)^{i+1} (v_{j_i}\ot 1 - 1\ot v_{j_i})\ot
   (v_{j_1}\wedge\cdots\wedge \hat{v}_{j_i}\wedge\cdots\wedge v_{j_p})
\end{equation}
for all $v_{j_1},\ldots, v_{j_p}\in V$
(e.g., see Weibel~\cite[\S4.5]{Weibel}).
We apply $\Hom_{S^e}(-, S\, g)$ to
each term of the Koszul resolution 
and then identify
$$
\Hom_{S^e}(S^e\ot\Wedge^{p}(V), S g)
\cong
\Hom_{k}(\Wedge^pV,S g)
\cong
S g\ot\Wedge^p V^*$$ 
for each $g$ in $G$.
Thus we write the set of {\bf cochains} arising from the Koszul resolution
(from which the cohomology classes emerge)
as vector forms on $V$ 
tagged by group elements:
Let 
\begin{equation}\label{C-cochains}
\CD=\bigoplus_{g\in G} C_g^{\hspace{.3ex}\DOT},
\quad\text{where}\quad 
C_g^p :=  S g \otimes \Wedge^{p} V^*
\quad  \  \text{ for each }\ g\in G.
\end{equation}
We call $\CD_g$ the space of cochains 
{\bf supported on} $g$.
Similarly, for any subset $X$ of $G$, we
define $\CD_X := \displaystyle{\oplus_{g\in X}}\CD_g$, 
the set of cochains
{\bf supported on} $X$. 
We say a cochain in $\CD$ is {\bf supported off}
a subset $X$ of $G$
if it lies in $\displaystyle{\oplus_{g\notin X}} \CD_g$.
Note that each element of $G$ permutes the summands 
of $C^{\hspace{.3ex}\DOT}$ via the conjugation action of $G$ on itself.

From the space $\CD$ of cochains, we define
a space of representatives of cohomology classes: Let
\begin{equation}\label{H-DOT}
\HD:=
\bigoplus_{g\in G}\ \ S(V^g) g \otimes \Wedge^{\DOT-\codim V^g}{(V^g)}^*\otimes
\Wedge^{\codim V^g}((V^g)^\perp)^* \ .
\end{equation}
Then $\HD\subset \CD$ with
$$
\HD\cong \HHD(S,S\#G)
\ \text{ and } \ 
(\HD)^G\cong \HHD(S,S\#G)^G\cong \HHD(S\#G) .
$$ (See~\cite[Proposition~5.11 and~(6.1)]{bracket} for this formulation of
the Hochschild cohomology. It was computed first independently by
Farinati~\cite{Farinati} and by Ginzburg and Kaledin~\cite{GinzburgKaledin}.)  
In particular it follows that $(H^2)^G$ is supported on elements $g$
for which $\codim V^g\in \{0,2\}$, since an element of $(H^2)^G$ is invariant
under the action of each group element $g$. 
See~\cite[Lemma~3.6]{SheplerWitherspoon1} for details. 

The grading on the polynomial ring $S=S(V)$ 
induces a grading on the set of cochains by
polynomial degree:
We say a cochain in $\CD$ has
{\bf polynomial degree}~$i$ if
the factors in $S$ in the expression (\ref{C-cochains}) are
all polynomials of degree $i$.  We say a cochain is 
{\bf homogeneous}
when its polynomial factors in $S$ are homogeneous.
A {\bf constant cochain} is then
one of polynomial degree $0$ 
and a {\bf linear cochain} is one of homogeneous polynomial degree $1$.
The cochains $\CD$ are filtered by polynomial degree:
$\CD_0\subset \CD_1\subset \CD_2\subset \cdots$,
where $\CD_i$ is the subspace of $\CD$ consisting of 
cochains of polynomial degree at most $i$.

\begin{definition}\label{bracketdef}
We define
a {\bf cochain bracket map}
on the subspace generated by 
linear and constant 2-cochains:
Let  $[*,*] :C_1^2\times C_1^2\rightarrow C_1^3$ 
be the symmetric map defined by $\hphantom{x}[\alpha,\beta](v_1,v_2, v_3)$
$$
:= \left\{
\begin{aligned}
& 
{\displaystyle \sum_{\substack{g,h\in G\\ \rule{0ex}{1.5ex}\sigma\in\Alt_3}}}\,
 \Big[
 \alpha_{gh^{-1}}\,
\left(\beta_h(v_{\sigma(1)}\wedge v_{\sigma(2)})\wedge v_{\sigma(3)}\right) 
+&
 \beta_{gh^{-1}}\,
\left(\alpha_h(v_{\sigma(1)}\wedge v_{\sigma(2)})\wedge v_{\sigma(3)}\right) 
\Big] g \\
& & \text{ for linear } \alpha,\beta
\\
& 
{\displaystyle \sum_{\substack{g,h\in G\\ \rule{0ex}{1.5ex}\sigma\in\Alt_3}}}\,
 \alpha_{gh^{-1}}\,
\left(\beta_h(v_{\sigma(1)}\wedge v_{\sigma(2)})\wedge v_{\sigma(3)}\right)
\ g 
& \text{ for constant } \alpha \text{ and linear } \beta  \\
& \rule{0ex}{4ex}
\quad0 & \quad\quad\text{ for constant } \alpha\text{ and }\beta
\end{aligned}
\right.
$$
for all $v_1, v_2, v_3$ in $V$. 
\end{definition}

We will see in the next section that this definition gives a representative
cochain for a class in cohomology $\HH^{\DOT}(S\#G)$ 
of the Gerstenhaber bracket of $\alpha$ and $\beta$ when they are cocycles.
\section{Gerstenhaber Bracket}\label{section:GB}

In this section we recall the definition of the Gerstenhaber bracket
on Hochschild cohomology, defined on the bar resolution, and show
how it is related to the cochain bracket map of Definition~\ref{bracketdef}.
Recall the definition of the 
{\bf bar resolution} of a $k$-algebra $R$:
It  has $p$th term $R^{\ot (p+2)}$ and 
differentials
$$
  \delta_p (r_0\ot\cdots\ot r_{p+1}) = \sum_{i=0}^p (-1)^i r_0\ot\cdots
     \ot r_ir_{i+1}\ot \cdots\ot r_{p+1}
$$
for all $r_0,\ldots, r_{p+1} \in R$.
From this one may derive the standard definition of a 
{\bf Hochschild 2-cocycle}: It is 
an element $\mu$ of $\Hom_k(R\ot R, R)\cong \Hom_{R^e}(R^{\ot 4},R)$ 
for which 
\begin{equation}\label{hoch2}
\mu(rs,u) + \mu(r,s)u = \mu(r,su) + r\mu(s,u) 
\end{equation}
for all $r,s,u\in R$. 
(Here we have further identified the linear map $\mu$ on $R\otimes R$
with a bilinear map on $R\times R$.)

We will need the following lemma for our calculations. 

\begin{lemma}\label{move-g}
Let $\mu$ be a Hochschild 2-cocycle on $S\# G$ whose kernel
contains $kG$. Then
$$
  \mu(rg,s) = \mu(r,gs) = \mu(r, {}^gs) g
$$
for all $r,s$ in $S$ and $g$ in $G$. 
\end{lemma}
\begin{proof}
Apply (\ref{hoch2}) to $r,g,s$ to obtain $\mu(rg,s)+\mu(r,g)s=
\mu(r,gs)+ r\mu(g,s)$. By hypothesis, $\mu(r,g)=0=\mu(g,s)$, so
$\mu(rg,s)=\mu(r,gs)$.
Now apply (\ref{hoch2}) to $r$, ${}^gs$, $g$ to obtain 
$\mu(r ({}^gs),g) + \mu(r, {}^gs)g = \mu(r, ({}^gs) g) + r\mu({}^gs, g)$.
By hypothesis, $\mu(r ({}^gs), g)=0=\mu( {}^gs , g)$,
so $\mu(r, {}^gs)g=\mu(r, ({}^gs) g)$.
Since $gs= ({}^gs) g$, the lemma follows.
\end{proof}
We will also need the definition of the circle operation 
on Hochschild cohomology in degree 2: If $R$ is a $k$-algebra and 
$\alpha$ and $\beta$ are elements of  $\Hom_{R^e}(R^{\ot 4}, R)
\cong \Hom_k (R^{\ot 2},R)$, then $\alpha\circ \beta\in\Hom_k(R^{\ot 3},R)$ is
defined by 
$$
   \alpha\circ\beta (r_1\ot r_2\ot r_3): = 
\alpha\big(\beta(r_1\ot r_2)\ot r_3\big)
   -\alpha\big(r_1\ot \beta(r_2\ot r_3)\big)
$$
for all $r_1,r_2,r_3\in R$.  The {\bf Gerstenhaber bracket} is then
$$
   [\alpha,\beta]:=
\alpha\circ\beta + \beta\circ\alpha.
$$
This bracket is well-defined on cohomology classes, however the circle
operation is not. 
In our setting, $R=S\# G$, and we now express the Gerstenhaber bracket on
input from the Koszul resolution using the cochain
bracket of Definition~\ref{bracketdef}.
In the theorem below, we fix
a choice of isomorphism $\HHD(S\# G)\cong (H^{\DOT})^G$ 
where $H^{\DOT}$ is given by (\ref{H-DOT}). 
(See~\cite[Proposition 5.11 and (6.1)]{bracket}.)

\begin{thm}\label{prop:2brackets}
Consider two cohomology classes $\alpha', \beta'$
in $\HH^2(S\# G)$ represented by cochains
$\alpha, \beta$ in $(H^2)^G$
of polynomial degree at most 1. 
Then the Gerstenhaber bracket in $\HH^3(S\#G)$
of $\alpha'$ and $\beta'$
is represented by the cochain bracket
$[\alpha,\beta]$ of Definition~\ref{bracketdef}.
\end{thm}
\begin{proof}
We use the chain map $\phi_{\DOT}$
from the Koszul resolution
$K_{\DOT}(S)$ to the bar resolution for $S=S(V)$ given in 
each degree by
\begin{equation}\label{eqn:phip}
  \phi_p(1\ot 1\ot v_{j_1}\wedge\cdots\wedge v_{j_p})
   = \sum_{\sigma\in\Sym_p} \sgn(\sigma) \ot v_{j_{\sigma(1)}}\ot
   \cdots\ot v_{j_{\sigma(p)}}\ot 1
\end{equation}
for all $v_{j_1},\ldots,v_{j_p}\in V$, where
$\Sym_p$ denotes the symmetric group on $p$ symbols. 
We may view functions on the bar resolution
in cohomological degree 2 as functions on $K_2(S)=S^e\ot
\Wedge^2(V)$ simply by composing with $\phi_2$.

We will also need a choice $\psi_{\DOT}$ of chain map from the
bar to the Koszul resolution.
The particular choice of $\psi_{\DOT}$ does not matter here, but
we will assume that $\psi\phi$ is the identity map
and that $\psi_2(1\ot a\ot b\ot 1) =0$ if either $a$ or $b$ is in 
the field $k$.
(For example, one could take $\psi_{\DOT}$ so that
$\psi_2(1\ot v_i\ot v_j\ot 1) = 
       1\ot 1\ot v_i\wedge v_j$
for $i<j$ and $0$ otherwise, for some fixed basis
$v_1,\ldots,v_n$ of $V$.
See~\cite{chainmap} for explicit constructions of such
maps $\psi_{\DOT}$; we will not need them here.) 
Note that although
$\psi_2$ may not be a $kG$-homomorphism, 
the map $\psi_2$ preserves the action of $G$
on the image of $\phi_2$.
For our purposes here, this implies that we do not need to average over $G$
when computing brackets, as is done, e.g., in~\cite{bracket}. 
We also note that every chain map $\psi$ for which $\psi\phi$ is the
identity map has the property that 
$$\psi_2(1\ot v\ot w\ot 1 
-1\ot w\ot v\ot 1) = 1\ot 1\ot v\wedge w$$ for all $v,w$ in $V$.
Thus on elements of this form, $\psi_2$ is independent of choice
of basis of $V$. 

We extend each function $\gamma$ on the bar complex for $S$ to a function
on the bar complex for $S\# G$ in a standard way: In degree 2, 
we require $kG$ to be in the kernel of $\gamma$ and set 
$$\gamma(1\ot s_1 g_1\ot s_2g_2\ot 1) := \gamma(1\ot s_1\ot {}^{g_1}\! s_2\ot 1) g_1g_2$$
for $\gamma\in \Hom_{S^e}(S^{\ot 4}, S\# G)$, $s_1,s_2\in S$, $g_1,g_2\in G$. 
Compare with Lemma~\ref{move-g}. 
(See~\cite[Theorem~5.4]{CGW} for a more general statement.) 

We apply the chain map $\psi$
to convert $\alpha$ and $\beta$ to functions on the bar complex,
execute the Gerstenhaber bracket, 
and apply $\phi$ to convert back to a function on the Koszul complex.
The induced operation  on cochains arising from
the Koszul complex is thus
$$
[\alpha , \beta]
:=\phi^*(\psi^*(\alpha)\circ\psi^*(\beta))
 + \phi^*(\psi^*(\beta)\circ \psi^*(\alpha))
\ .
$$
Note there is no guarantee that $[\alpha,\beta]$ is in the chosen space
$H^3$ of representatives of cohomology classes, however there is a 
unique element of $H^3$ to which it is cohomologous. 

We compute separately the two corresponding circle operations,
keeping in mind that they are not well-defined on cohomology,
and so must be combined. (Again, we identify
$\text{Hom}_{S^e}(S^{\otimes (p+2)}, *)$ with $\text{Hom}_{k}(S^{\otimes p}, *)$
and $\text{Hom}_{S^e}(S^e\otimes \Wedge^pV, *)$
with $\text{Hom}_{k}(\Wedge^pV,*)$, dropping
extra tensor factors of $1$.)
Then $$
\begin{aligned}
  (\alpha\circ\beta)&(v_1\wedge v_2\wedge v_3)\\
&=  \big(\psi^*(\alpha)\circ \psi^*(\beta)\big) 
\phi(v_1\wedge v_2\wedge v_3)\\ 
&=\big(\psi^*(\alpha)\circ \psi^*(\beta)\big) 
\sum_{\sigma \in S_3} \text{sgn}(\sigma) \
v_{\sigma(1)}\ot v_{\sigma(2)}\ot v_{\sigma(3)}\\
&=
\sum_{\sigma \in S_3} \text{sgn}(\sigma) \
\psi^*(\alpha)
\Big(\psi^*(\beta)(v_{\sigma(1)}\ot v_{\sigma(2)})\ot v_{\sigma(3)}
-
v_{\sigma(1)}\ot \psi^*(\beta)(v_{\sigma(2)}\ot v_{\sigma(3)})\Big).
\end{aligned}
$$
We may rewrite the sum over the alternating
group instead to obtain
$$
\begin{aligned}
  (\alpha\circ\beta)(v_1\wedge v_2\wedge v_3)\\
=
\sum_{\sigma \in \text{Alt}_3} \
& \psi^*(\alpha)
\Big(\psi^*(\beta)
(v_{\sigma(1)}\ot v_{\sigma(2)})\ot v_{\sigma(3)}
-
\psi^*(\beta)(v_{\sigma(2)}\ot v_{\sigma(1)})\ot v_{\sigma(3)}\Big)\\
-& 
\psi^*(\alpha)
\Big(v_{\sigma(1)}\ot \psi^*(\beta)(v_{\sigma(2)}\ot v_{\sigma(3)})
-
v_{\sigma(1)}\ot \psi^*(\beta)(v_{\sigma(3)}\ot v_{\sigma(2)})\Big).
\end{aligned}
$$
But $\psi^*(\beta)(v\tensor w - w\tensor v) = \beta(v\wedge w)$
for all vectors $v,w$ in $V$,
and hence Lemma~\ref{move-g} implies that the above sum is just
\begin{equation}\label{yuck}
\begin{aligned}
\sum_{\sigma \in \text{Alt}_3} \ \ \ \ 
&\psi^*(\alpha)
\Big(
\beta(v_{\sigma(1)}\wedge v_{\sigma(2)})\ot v_{\sigma(3)}
-
v_{\sigma(1)}\ot \beta(v_{\sigma(2)}\wedge v_{\sigma(3)})\Big)\\
=
\sum_{\sigma \in \text{Alt}_3,\ h \in G} \
&\psi^*(\alpha)
\Big(
\beta_h(v_{\sigma(1)}\wedge v_{\sigma(2)})h\ot v_{\sigma(3)}
-
v_{\sigma(1)}\ot \beta_h(v_{\sigma(2)}\wedge v_{\sigma(3)})h\Big)\\
=
\sum_{\sigma \in \text{Alt}_3,\ h \in G} \
&\psi^*(\alpha)
\Big(
\beta_h(v_{\sigma(1)}\wedge v_{\sigma(2)})\ot {}^h v_{\sigma(3)}
-
v_{\sigma(1)}\ot \beta_h(v_{\sigma(2)}\wedge v_{\sigma(3)})\Big)\ h\ .
\end{aligned}
\end{equation}

First assume the polynomial degree of
$\beta$ is $0$.
Then each $\beta_h(v_i\wedge v_j)$ is constant.
But
$\psi^*(\alpha)\big(a\ot b\big)$
is zero for either $a$ or $b$ in $k$,
and the last expression is thus zero.
Hence,
$\alpha\circ\beta(v_1\wedge v_2\wedge v_3)$
is zero for $\beta$ of polynomial degree 0.

Now assume $\beta$ is homogenous of polynomial degree $1$.
We claim that for
any $h$ in $G$ and any $u_1, u_2, u_3$ in $V$,
\begin{equation}\label{h-noth}
\sum_{\sigma\in \text{Alt}_3}
\beta_h(u_{\sigma(1)}\wedge u_{\sigma(2)})\ot {}^h u_{\sigma(3)}
=
\sum_{\sigma\in \text{Alt}_3}
\beta_h(u_{\sigma(1)}\wedge u_{\sigma(2)})\ot u_{\sigma(3)}.
\end{equation}
The equation clearly holds for $h$ acting trivially on $V$.
One may easily verify the equation for $h$ not in the kernel of
the representation $G\rightarrow\text{GL}(V)$ by fixing a basis
of $V$ consisting of eigenvectors for $h$ and using the fact that
any nonzero $\beta_h$ is supported on $\bigwedge^2(V^h)^\perp$
with $\codim V^h=2$; see~(\ref{H-DOT}).

We use Equation~\ref{h-noth} and the fact that
$\psi^*(\alpha)(v\tensor w - w\tensor v) = \alpha(v\wedge w)$
for all vectors $v,w$ in $V$ to simplify Equation~\ref{yuck}:
$$
\begin{aligned}
(\alpha\circ\beta)&(v_1\wedge v_2\wedge v_3)\\
& =
\sum_{\sigma \in \text{Alt}_3,\ h \in G} \
\psi^*(\alpha)
\Big(
\beta_h(v_{\sigma(1)}\wedge v_{\sigma(2)})\ot {}^h v_{\sigma(3)}
-
v_{\sigma(1)}\ot \beta_h(v_{\sigma(2)}\wedge v_{\sigma(3)})\Big)\ h\ \\
&=
\sum_{\sigma \in \text{Alt}_3,\ h \in G} \
\psi^*(\alpha)
\Big(
\beta_h(v_{\sigma(1)}\wedge v_{\sigma(2)})\ot v_{\sigma(3)}
-
v_{\sigma(1)}\ot \beta_h(v_{\sigma(2)}\wedge v_{\sigma(3)})\Big)\ h\ \\
&=
\sum_{\sigma \in \text{Alt}_3,\ h \in G} \
\psi^*(\alpha)
\Big(
\beta_h(v_{\sigma(1)}\wedge v_{\sigma(2)})\ot v_{\sigma(3)}
-
v_{\sigma(3)}\ot \beta_h(v_{\sigma(1)}\wedge v_{\sigma(2)})\Big)\ h\ \\
&=
\sum_{\sigma \in \text{Alt}_3,\ h \in G} \
\alpha
\Big(
\beta_h(v_{\sigma(1)}\wedge v_{\sigma(2)})\wedge v_{\sigma(3)}\Big)\ h\\
&=
\sum_{\sigma \in \text{Alt}_3;\ g,h \in G} \
\alpha_g
\Big(
\beta_h(v_{\sigma(1)}\wedge v_{\sigma(2)})\wedge v_{\sigma(3)}\Big)\ gh \ .
\end{aligned}
$$
A similar computation for $\beta \circ \alpha$ together with 
reindexing over the group yields the result.
\end{proof}

The Gerstenhaber bracket takes a particularly nice form when we
consider square brackets of linear cocycles and brackets of
linear with constant cocycles:

\begin{cor}\label{niceform}
Consider cohomology 
classes $\alpha', \beta'$
in $\HH^2(S\# G)$ represented respectively by
a constant cocycle
$\alpha$ and a linear cocycle $\beta$ in $(H^2)^G$.
The Gerstenhaber bracket in $\HH^3(S\#G)$
of $\beta'$ with itself and of $\alpha'$ with $\beta'$
are represented by the cocycles 
$$
\hphantom{x}[\beta,\beta](v_1,v_2, v_3)
= 2
\sum_{\substack{g,h\in G\\\rule{0ex}{1.5ex}\sigma\in\Alt_3}}\,
 \beta_{gh^{-1}}\,
\left(\beta_h(v_{\sigma(1)}\wedge v_{\sigma(2)})\wedge v_{\sigma(3)}\right) 
\ g
$$
and
$$
\hphantom{x}[\alpha,\beta](v_1,v_2, v_3)\\
= 
\sum_{\substack{g,h\in G\\\rule{0ex}{1.5ex}\sigma\in\Alt_3}}\,
\alpha_{gh^{-1}}\,
\left(\beta_h(v_{\sigma(1)}\wedge v_{\sigma(2)})\wedge v_{\sigma(3)}\right)
\ g\ ,
$$
respectively.
\end{cor}

\section{PBW Condition and Gerstenhaber Bracket}\label{sec:PBW}

In this section, we give necessary and sufficient conditions
on a parameter to define a Drinfeld orbifold algebra
in terms of Hochschild cohomology.
We interpret Theorem~\ref{thm:4condns} 
in terms of cocycles and the Gerstenhaber bracket
in cohomology as realized on the set of cochains arising
from the Koszul resolution.
Our results should be compared with~\cite[\S2.2, (4), (5), (6)]{HOT},
where a factor of 2 is missing from the right side of (5).
See also~\cite[(1.9)]{Khare} for a somewhat different setting.

We want to describe precisely which parameter maps $\kappa$ result in a quotient $\cH_{\kappa}$
that satisfies the PBW condition, that is, defines a Drinfeld orbifold algebra.   
The algebras $\cH_{\kappa}$ are 
naturally expressed and analyzed in terms of the Koszul resolution of $S$.
Recall,  $\kappa: \Wedge^2V \rightarrow S\otimes \CC G$
with $\kappa = \sum_{g\in G}\kappa_g g$.
The parameter map $\kappa$ as well as its
linear and constant parts, $\kappa^L$ and $\kappa^C$, thus define
cochains on the Koszul resolution
and we identify $\kappa, \kappa^L, \kappa^C$ 
with elements of $C^{\hspace{.3ex}\DOT}$. Indeed, 
for each $g\in G$, the functions $\kappa_g g$, $\kappa_g^L g$, 
and $\kappa_g^C g$ (from $\Wedge^2 V$ to $S g$)
define elements 
of the cochain complex $C_g^{\hspace{.3ex}\DOT}$ of (\ref{C-cochains}).

We now determine a complete set of
necessary and sufficient conditions on these parameters
regarded as cochains in Hochschild cohomology
$\HHD(S\# G) \cong\HHD(S,S\#G)^G$.
(We use the chain maps converting between resolutions
discussed in Section~\ref{section:GB}.)
The significance of the following lemma and theorem 
thereafter lies in the expression of
the PBW property in terms of the Gerstenhaber bracket in cohomology.
We distinguish a cochain $[\alpha,\beta]$
on the Koszul resolution~(\ref{koszul-res2}) 
from its
cohomology class arising from the induced Gerstenhaber bracket
by using the phrase ``as a cochain'' 
where appropriate.
(Recall that $d$ is the differential on the Koszul resolution defined as in
 (\ref{koszul-diff}).)

\begin{lemma}\label{prop:3condns}
In Theorem~\ref{thm:4condns}, 
\begin{itemize}
\item
Condition~(ii)
holds if and only if $\kappa^L$ is a cocycle,
i.e., $d^*\kappa^L = 0$.
\item
For $\kappa^L$ in $H^{\DOT}$,
Condition~(iii) is equivalent to $[\kappa^L,\kappa^L] = 2 d^*\kappa^C$
as cochains. 
\item
For $\kappa^L$ in $H^{\DOT}$,
Condition~(iv) is equivalent to $[\kappa^C, \kappa^L]=0$ as a cochain.
\end{itemize}
\end{lemma}

\begin{proof}
The cochain $d^*\kappa^L$ is zero exactly when $\kappa^L$ takes to 0 all
input of the form
$$
\begin{aligned}
 d_3(v_1\wedge v_2\wedge v_3) & = (v_1\ot 1-1\ot v_1)\ot v_2\wedge v_3
      - (v_2\ot 1 - 1\ot v_2)\ot v_1\wedge v_3 \\
     & \hspace{5cm} + (v_3\ot 1 -1\ot v_3)\ot v_1\wedge v_2\, ,
\end{aligned}
$$
in other words, when
$$
  0 = v_1\kappa^L(v_2,v_3)-\kappa^L(v_2,v_3)v_1
   + v_2\kappa^L(v_3,v_1)- \kappa^L(v_3,v_1)v_2 
    + v_3\kappa^L(v_1,v_2) - \kappa^L(v_1,v_2)v_3
$$
in $S\#G$.
This is equivalent to
$$ 
\begin{aligned}
0\ =\ &\ v_1\kappa^L_g(v_2,v_3)g - \kappa^L_g(v_2,v_3)gv_1
    + v_2\kappa^L_g(v_3,v_1)g \\
&-\kappa^L_g(v_3,v_1)gv_2 
    + v_3\kappa^L_g(v_1,v_2)g - \kappa^L_g(v_1,v_2)gv_3
\ 
\end{aligned}
$$
for each $g$ in $G$.
We rewrite this expression using the commutativity of $S$ and moving
all factors of $g$ to the right:
$$
  0 = \kappa^L_g(v_2,v_3)( v_1 - {}^gv_1)
    + \kappa^L_g(v_3,v_1) ( v_2 - {}^gv_2)
    + \kappa^L_g(v_1,v_2) ( v_3 - {}^gv_3),
$$
which is precisely Theorem~\ref{thm:4condns}(ii).

Next, notice that we may apply 
Equation~\ref{h-noth} (in the proof of
Theorem~\ref{prop:2brackets}) to
$\beta=\kappa^L$, under the assumption that $\kappa^L$ lies in $H^2$.
Then for each $g$ in $G$,
the left side of Theorem~\ref{thm:4condns}(iii)
is the opposite of the coefficient of $g$
in $[\kappa^L,\kappa^L]$ by Definition~\ref{bracketdef} 
(see~Corollary~\ref{niceform}) and the skew-symmetry of $\kappa^L$.
By a similar calculation to that for $\kappa^L$, the right side of 
Theorem~\ref{thm:4condns}(iii) is the coefficient of $g$ in
$$
   -2 d_3^*\kappa^C(v_1\wedge v_2\wedge v_3)
   = 2 \sum_{g\in G}\sum_{\sigma\in\Alt_3} \kappa^C_g(v_{\sigma(2)},v_{\sigma(3)})
   ( {}^g v_{\sigma(1)} - v_{\sigma(1)}) g .
$$
Hence, Theorem~\ref{thm:4condns}(iii) 
is equivalent to $[\kappa^L,\kappa^L] = 2d_3^*\kappa^C$.
This condition differs from~\cite[(5)]{HOT} where the factor
of 2 is missing.

We again compare coefficients of fixed $g$ in $G$ and apply
Equation~\ref{h-noth} to see that
Theorem~\ref{thm:4condns}(iv) is equivalent to $[\kappa^C,\kappa^L]=0$
by Definition~\ref{bracketdef}. Note that this is 
equivalent to~\cite[(6)]{HOT} when $k={\mathbb R}$.
\end{proof}

We are now ready to express the PBW property purely in cohomological terms.
Recall that $H^{\DOT}$ is the fixed space of representatives of elements 
in $\HH^{\DOT}(S, S\# G)$ defined in (\ref{H-DOT}). 
\begin{thm}
\label{4Hochschildconditions}
A quotient algebra $\cH_{\tilde{\kappa}}$ 
is a Drinfeld orbifold algebra if and only if
$\cH_{\tilde{\kappa}}$ is isomorphic
to $\cH_{\kappa}$ as a filtered algebra 
for some parameter $\kappa$ satisfying
\begin{itemize}
\item[(i)] $\kappa$ is $G$-invariant,
\item[(ii)] The linear part
of $\kappa$ is a cocycle in $H^{\DOT}$,
\item[(iii)] The Gerstenhaber square bracket of the linear part
of $\kappa$ satisfies 
$[\kappa^L, \kappa^L] = 2 d^*(\kappa^C)$ as cochains,
\item[(iv)] The bracket of the linear with the constant part
of $\kappa$ is zero:
$[\kappa^C, \kappa^L]=0$ as a cochain.
\end{itemize}
\end{thm}
\begin{proof}
Write $\tilde{\kappa}=\tilde{\kappa}^L+\tilde{\kappa}^C$.
If $\cH_{\tilde{\kappa}}$ is a Drinfeld orbifold algebra,
then it satisfies Condition~(ii) of Theorem~\ref{thm:4condns}.  
Lemma~\ref{prop:3condns} then implies
that $\tilde{\kappa}^L$ is a cocycle in $\HH^{\DOT}(S,S\#G)$ expressed
with respect to the Koszul resolution, and it thus lies in the set
of cohomology representatives $(H^{\DOT})^G$ up to a coboundary:
$$
\tilde{\kappa}^L=\kappa^L+d^*\rho
$$
for some 2-cocycle $\kappa^L$ in $(H^2)^G$ and some 1-cochain $\rho$.
Set $$\kappa^C= \tilde{\kappa}^C + \rho \circ \tilde{\kappa}^L -\rho\ot\rho$$ 
where $(\rho\ot\rho) (v\wedge w) := \rho(v)\rho(w) - \rho(w)\rho(v)$ for all $v,w\in V$.
Let $\kappa=\kappa^C+\kappa^L$.  

We may assume without loss of generality that $\rho$ is $G$-invariant.
(Note that $d^*\rho = \tilde{\kappa}^L - \kappa^L$ is
$G$-invariant. Since $d^*$ commutes with the group action and the order
of $G$ is invertible in $k$, we may replace $\rho$ by $\frac{1}{|G|}
\sum_{g\in G} {}^g \rho$ to obtain a cochain having the same image under $d^*$.)
Also note that without loss of generality $\rho$ takes values in $kG$ since
$d^*\rho$ has polynomial degree 1.

Define a map $f: T(V)\# G \rightarrow \cH_{\kappa}$ by
$$
  f(v) = v+\rho(v), \ \ \ f(g) = g
$$
for all $v\in V$, $g\in G$; since $\rho$ is $G$-invariant, these values extend
uniquely to give an algebra homomorphism. 
Note that $f$ is surjective by an inductive argument on the degrees of elements.

We show first that the kernel of $f$
contains the ideal 
$$
  (vw - wv - \tilde{\kappa}^L(v,w) -\tilde{\kappa}^C(v,w) \mid v,w\in V ) ,
$$
which implies that $f$ induces an algebra homomorphism from $\cH_{\tilde{\kappa}}$
onto $\cH_{\kappa}$. 
By the definition of $f$, 
$$
\begin{aligned}
& f(vw-wv-\tilde{\kappa}^L(v,w) -\tilde{\kappa}^C(v,w)) \\
& \ \ = \ (v+\rho(v))(w+\rho(w)) - (w+\rho(w))(v+\rho(v)) \\
 & \ \hspace{2cm} - \tilde{\kappa}^L(v,w) - \rho(\tilde{\kappa}^L(v,w)) - \tilde{\kappa}^C(v,w)\\
 & \ \ = \ vw-wv + v\rho(w) + \rho(v)w -w\rho(v)-\rho(w)v +\rho(v)\rho(w) -\rho(w)\rho(v)\\
 & \ \hspace{2cm} -\tilde{\kappa}^L(v,w) - \rho\circ\tilde{\kappa}^L(v,w) -\tilde{\kappa}^C(v,w) \\
& \ \ = \ vw-wv + d^*\rho(v\wedge w) + (\rho\ot\rho)(v\wedge w) 
   -\tilde{\kappa}^L(v,w) -\rho\circ\tilde{\kappa}^L(v,w) -\tilde{\kappa}^C(v,w)\\
  & \ \ = \ vw-wv -\kappa^L(v,w) - \kappa^C (v,w)  \ =  \ 0
\end{aligned}
$$
in $\cH_{\kappa}$. Thus the ideal generated by all $vw-wv-\tilde{\kappa}^L(v,w)
-\tilde{\kappa}^C(v,w)$ is in the kernel of $f$.

Next we define an inverse to $f$ by replacing
$\rho$ with $-\rho$:
Define an algebra homomorphism
$f' : T(V) \# G \rightarrow \cH_{\tilde{\kappa}}$ by
$f'(v) = v - \rho(v)$, $f(g)=g$ for all $v\in V$, $g\in G$.
We have $\kappa^L = \tilde{\kappa}^L - d^*\rho$ and
\begin{eqnarray*}
  \tilde{\kappa}^C & = & \kappa^C - \rho\circ \tilde{\kappa}^L + \rho\ot\rho\\
    & = & \kappa^C - \rho\circ \kappa^L - \rho\circ (d^*\rho) + \rho\ot\rho.
\end{eqnarray*}
Extending $\rho$ in the usual way from a function on $V$ to a function on $V\ot kG$
by setting $\rho(vg) := \rho(v) g$ for all $v\in V$, $g\in G$, we calculate
\begin{eqnarray*}
  \rho\circ (d^*\rho) (v\wedge w) & = & \rho(v \rho(w) + \rho(v)w - w\rho(v) -\rho(w) v)\\
   & = & \rho(v)\rho(w) + \rho(v)\rho(w) - \rho(w)\rho(v) - \rho(w)\rho(v) \\
 & = & 2(\rho\ot \rho)(v\wedge w),
\end{eqnarray*}
since $\rho$ has image in $kG$.
Thus we may rewrite 
\begin{eqnarray*}
 \tilde{\kappa}^C & = & \kappa^C - \rho\circ \kappa^L - 2\rho\ot\rho + \rho\ot \rho\\
   & = & \kappa^C - \rho\circ\kappa^L - \rho\ot\rho\\
  &= & \kappa^C + (-\rho)\circ \kappa^L - (-\rho)\ot (-\rho). 
\end{eqnarray*} 
An argument similar to that above for $f$ (replacing $\rho$ by $-\rho$)
shows that
the function $f'$
induces an algebra homomorphism from $\cH_{\kappa}$ onto $\cH_{\tilde{\kappa}}$. 
By its definition, $f'$ is inverse to $f$. Therefore $\cH_{\kappa}$ and $\cH_{\tilde{\kappa}}$
are isomorphic as filtered algebras. 
(Note that this isomorphism did 
not require that $\cH_{\tilde{\kappa}}$
satisfy the PBW condition, only that
$\tilde{\kappa}$ be a cocycle.) 
As $\gr \cH_{\kappa}
\cong \gr \cH_{\tilde{\kappa}} \cong S\# G$, the quotient algebra
$\cH_{\kappa}$ is also a Drinfeld orbifold algebra. 
Theorem~\ref{thm:4condns} and 
Lemma~\ref{prop:3condns} then imply the four conditions of the theorem.

Conversely, assume $\cH_{\tilde{\kappa}}$ is isomorphic, as a filtered algebra,
to some $\cH_{\kappa}$ satisfying the four conditions of the theorem.
Then Theorem~\ref{thm:4condns} and Lemma~\ref{prop:3condns} 
imply that $\cH_{{\kappa}}$
is a Drinfeld orbifold algebra.  As the isomorphism preserves the filtration,
$$\gr\cH_{\tilde{\kappa}}\cong \gr\cH_{\kappa} \cong S(V)\#G$$
as algebras, and hence $\cH_{\tilde{\kappa}}$ is a Drinfeld orbifold algebra as well.
Note that Theorem~\ref{prop:2brackets} shows that the bracket
formula in the statement of the theorem indeed coincides with the
Gerstenhaber bracket on cohomology.
\end{proof}

\begin{remark}
{\em 
We compare the above results to Gerstenhaber's original theory
of deformations, since
every Drinfeld orbifold algebra
defines a deformation of $S\# G$ (see Section~\ref{deformations}).
The theory of Hochschild cohomology provides necessary conditions
for ``parameter maps'' to define a deformation.
Given a $k$-algebra $R$ and arbitrary $k$-linear maps 
$\mu_1, \mu_2: R\otimes R\rightarrow R$, we say $\mu_1$
and $\mu_2$ {\em extend} to first and second order approximations, respectively, of 
a deformation $R[t]$ of $R$ over $k[t]$ if there are $k$-linear maps
$\mu_i: R\otimes R\rightarrow R$ ($i\geq 3$) for which
the multiplication in $R[t]$ satisfies
$$
  r*s = rs + \mu_1(r\ot s)t + \mu_2(r\ot s) t^2 + \mu_3(r\ot s)t^3 + \cdots
$$
for all $r,s\in R$, where $rs$ is the product in $R$.
Associativity forces $\mu_1$ to define a cocycle in $\HH^2(R)$;
in addition, its Gerstenhaber square bracket
must be twice the differential applied to $\mu_2$:
$$
[\mu_1, \mu_1]
 = 2\delta_3^* \mu_2\ .
$$
Indeed, by using (\ref{eqn:kappa-mu})
and (\ref{eqn:kappa-mu2}), we find that the
equation $[\kappa^L, \kappa^L] = 2d^* \kappa^C$ is a consequence of 
the equation $[\mu_1,\mu_1]= 2 \delta_3^*\mu_2$:
The left side of Theorem~\ref{thm:4condns}(iii) is both equal to 
$-[\kappa^L,\kappa^L]$
applied to $v_1\wedge v_2\wedge v_3$ and to 
$-[\mu_1,\mu_1]$ applied to $v_1\wedge v_2\wedge v_3$
by our previous analysis, identifying $\alpha$ in the Braverman-Gaitsgory approach with
the restriction of $\mu_1$ to the space of relations $R$. 
The right side of Theorem~\ref{thm:4condns}(iii) is both 
equal to $-2d^*\kappa^C$ applied
to $v_1\wedge v_2\wedge v_3$ and to $-2\delta_3^*\mu_2$
applied to $v_1\wedge v_2\wedge v_3$ since $\mu_2\circ \phi_2 = \kappa^C$
and $\phi$ is a chain map (see (\ref{eqn:phip})). 

The square bracket $[\mu_1, \mu_1]$
is called the {\em primary obstruction} to integrating a map $\mu_1$ 
to a deformation:
If a deformation exists with first-order approximation $\mu_1$, then
$[\mu_1, \mu_1]$ is a coboundary, i.e., defines the 
zero cohomology class of the
Hochschild cohomology $\HH^3(R)$.

The parameter maps $\kappa^L$ and $\kappa^C$ 
(arising from the Koszul resolution)
play the role 
of the first and second order approximation maps $\mu_1$ and $\mu_2$
(arising from the bar complex).
We see in the proof of Theorem~\ref{degreemu_1} that each $\kappa_g^L g$
is in fact a {\em cocycle} 
when $\cH_{\kappa}$ is a Drinfeld orbifold algebra, and each $\kappa^C_g g$ 
defines a second order approximation to the deformation.
In fact, we expect $\kappa^L$ to be invariant 
whenever $\cH_{\kappa}$ is a Drinfeld orbifold algebra
since $\HHD(S\#G)\cong \HHD(S,S\#G)^G$.
Note however that the theorem above goes beyond these elementary observations
and Gerstenhaber's original formulation, which
only give necessary conditions.
}
\end{remark}

\vspace{2ex}

We now apply Theorem~\ref{4Hochschildconditions} in special cases
to determine Drinfeld orbifold algebras from the set of
necessary and sufficient conditions given in that theorem
(in terms of Gerstenhaber brackets).

Recall that the Lie orbifold algebras 
are exactly the PBW algebras $\cH_{\kappa}$
in which the linear part of the parameter $\kappa$
is supported on the identity group element $1_G$ alone.  
Interpreting Proposition~\ref{PBW} in homological language,
we obtain necessary and sufficient conditions for $\kappa$
to define a Lie orbifold algebra in terms of the Gerstenhaber bracket:

\begin{cor}
Assume $\kappa^L$ is supported on $1_G$. Then
$\cH_{\kappa}$ is a Lie orbifold algebra if and only if
\begin{itemize}
\item[(a)] $\kappa^L$ is a Lie bracket on $V$,
\item[(b)] both $\kappa^L$ and $\kappa^C$ are
$G$-invariant cocycles (define elements of $\HH^2(S\#G)$),
\item[(c)] 
$[\kappa^C, \kappa^L]=0$ as a cochain.
\end{itemize}
\end{cor}
\begin{proof}
First note that
$[\kappa^L,\kappa^L]=0$ exactly when $\kappa^L$ defines a Lie bracket
on $V$.
Suppose Conditions~(a), (b), and (c) hold.
Condition~(b) implies parts~(i) and (ii) of Theorem~\ref{4Hochschildconditions}.
It also implies that $d^*(\kappa^C)=0$.
Condition~(a) implies that $[\kappa^L, \kappa^L]=0$,
and part (iii) of Theorem~\ref{4Hochschildconditions} is satisfied as well.
Condition~(c) is part~(iv) of Theorem~\ref{4Hochschildconditions}.
Hence, Theorem~\ref{4Hochschildconditions}
implies that $\cH_{\kappa}$ is a Lie orbifold algebra.

Conversely,  assume that $\cH_{\kappa}$ is a Lie orbifold algebra.
By Proposition~\ref{PBW},
$\kappa^L$ defines a Lie bracket on $V$
and hence $[\kappa^L,\kappa^L]=0$.
Theorem~\ref{4Hochschildconditions} then not only implies 
Condition~(c), but also that
$\kappa^L$ and $\kappa^C$ are both cocycles with $\kappa$ $G$-invariant.
But
$\kappa$ is $G$-invariant if and only if both
$\kappa^L, \kappa^C$ are $G$-invariant.
Hence, Condition~(b) holds.
\end{proof}

Recall that $(\HD)^G\cong \HHD(S\#G)$
 and that $\CD$ and $\HD$ are sets of cochains
and cohomology representatives, respectively (see~(\ref{C-cochains}) and~(\ref{H-DOT})). 
Given $\kappa^C,\kappa^L$ in $(H^2)^G$ of homogeneous polynomial
degrees $0$ and $1$, respectively,
the sum $\kappa:=\kappa^C+\kappa^L$ is a parameter function
$V\wedge V\rightarrow (k\oplus V)\otimes kG$
defining a quotient algebra $\cH_{\kappa}$.
The last result implies immediately
that for
$\kappa^L$ supported on $1_G$,  
the algebra $\cH_{\kappa}$ is a Lie orbifold algebra when
 $\kappa^L$ is a Lie bracket on $V$ and the cochain $[\kappa^C, \kappa^L]$ is zero
on the Koszul resolution. 
The hypothesis that $\kappa^L$ be a Lie bracket is not as restrictive
as one might think.  In fact, if $\kappa^L$ is a noncommutative Poisson structure
(i.e., with Gerstenhaber square bracket $[\kappa^L,\kappa^L]$ zero
in cohomology), then $\kappa^L$ is automatically a Lie bracket, as we see in the next corollary.

\begin{cor}\label{offkernel}
Suppose a linear cochain $\kappa^L$ in $C^2$ 
is supported on the kernel of the representation $G\rightarrow\text{GL}(V)$
and that $[\kappa^L,\kappa^L]$ is a coboundary.
Then $[\kappa^L,\kappa^L]=0$ as a cochain.
\end{cor}
\begin{proof}
Suppose $[\kappa^L,\kappa^L]=d^* \alpha$
for some $\alpha$.  Then (by definition of the map $d^*$),
$$
   d_3^*\alpha(v_1\wedge v_2\wedge v_3)
   = - \sum_{g\in G}\sum_{\sigma\in\Alt_3} \alpha_g(v_{\sigma(2)},v_{\sigma(3)})
   ( {}^g v_{\sigma(1)} - v_{\sigma(1)}) g 
$$ 
for all $v_1,v_2, v_3$ in $V$, and thus $d^*\alpha$ is supported off the
kernel $K$ of the representation $G\rightarrow\text{GL}(V)$.
But by Definition~\ref{bracketdef}, $[\kappa^L,\kappa^L]$ is supported 
on $K$, since $\kappa^L$ itself is supported on $K$.
Hence $[\kappa^L,\kappa^L]$ must be the zero cochain.
\end{proof}

The last corollary implies that every
linear noncommutative Poisson structure supported on
group elements acting trivially
lifts (or integrates) to a deformation of $S\# G$: 
\begin{cor}
Suppose a linear cocycle $\kappa^L$ in $(H^2)^G$ has trivial
Gerstenhaber square bracket in cohomology.
If $\kappa^L$ 
is supported on the kernel of the representation $G\rightarrow\text{GL}(V)$,
then the quotient algebra 
$\cH_{\kappa}$ with $\kappa=\kappa^L$
is a Drinfeld orbifold algebra.
Moreover, if $G$ acts faithfully on $V$, then $\cH_{\kappa}
\cong \mathcal{U}(\g)\#G$, a Lie orbifold algebra.
\end{cor}
\begin{proof}
Since $\kappa^L$ lies in $(H^2)^G$, we may set
$\kappa^C\equiv 0$ and $\kappa:=\kappa^L$ to satisfy the 
conditions of Theorem~\ref{4Hochschildconditions}
(using Corollary~\ref{offkernel}
to deduce that $\kappa^L$ is a Lie bracket).
If $G$ acts faithfully, the resulting Drinfeld orbifold algebra is just the skew group algebra
$\mathcal{U}(\g)\# G$, where the Lie algebra
$\g$ is the vector space $V$ with Lie bracket $\kappa^L$.
\end{proof}

\begin{remark}{\em 
The analysis of the Gerstenhaber bracket in~\cite{bracket}
includes information on the case of cocycles supported {\em off}
the kernel $K$ of the representation $G\rightarrow\text{GL}(V)$.
Indeed, we see in~\cite{bracket}
that if $\kappa^L$ in $(H^2)^G$ is supported off $K$,
then $[\kappa^L,\kappa^L]$ is always a coboundary.  
This guarantees existence of a constant cochain $\kappa^C$ with 
$[\kappa^L,\kappa^L]=2d^*\kappa^C$.
Thus to satisfy the conditions of Theorem~\ref{4Hochschildconditions},
one need only check that
$[\kappa^C,\kappa^L]=0$ as a cochain (on the Koszul resolution).

On the other hand, if $\kappa^L$ in $H^{\DOT}$ is supported {\em on}
the kernel $K$, and $[\kappa^L,\kappa^L]$ is a coboundary, then by
Corollary~\ref{offkernel}, $[\kappa^L,\kappa^L] = 0$ as a cochain.
Thus to satisfy the conditions of Theorem~\ref{4Hochschildconditions}, one
need only solve the equation $[\kappa^C,\kappa^L]=0$ as a cochain
for $\kappa^C$ a {\em cocycle}. 
}
\end{remark}


\section{Applications to Abelian Groups}\label{sec:abel}
The last section expressed the PBW condition in terms of simple conditions on 
Hochschild cocycles.  We see in this section how this alternative
formulation gives a quick and clear proof that 
every linear noncommutative Poisson structure 
(i.e., Hochschild 2-cocycle
with trivial Gerstenhaber square bracket)
lifts to a deformation
when $G$ is abelian. 
Halbout, Oudom, and Tang  
\cite[Theorem 3.7]{HOT} gave an analogous result
 over the real numbers for arbitrary groups (acting faithfully), but their proof
does not directly extend to other fields such as the complex numbers.
(For example, complex reflections in a finite group acting linearly on $\CC^n$
may contribute to Hochschild cohomology $\HH^2(S\#G)$ defined
over the real numbers, but not to the same cohomology
defined over the complex numbers.)

In the case of nonabelian groups,
the square
bracket of the linear part of the parameter $\kappa$
may be zero in cohomology but nonzero as a cochain.
The following proposition explains that
this complication disappears for abelian groups:

\begin{prop}\label{prop:abelian2}
Let $G$ be an abelian group.
Let $\alpha,\beta$ in $(H^2)^G$ be linear 
with Gerstenhaber
bracket $[\alpha,\beta]$ a coboundary (defining the zero cohomology class).
Then $[\alpha,\beta]=0$ as a cochain.
\end{prop}
\begin{proof}
Let $v_1,\ldots,v_n$ be a basis of $V$ on which $G$ acts diagonally.
If $[\alpha,\beta]$ is nonzero at the chain level, 
then some summand of Definition~\ref{bracketdef} is nonzero for
some triple $v_1,v_2,v_3$.
Suppose without loss of generality that 
$$
w=\beta_h\,
\left(v_3\wedge \alpha_g(v_1\wedge v_2)\right)
\ 
$$
is nonzero for some $g,h$ in $G$.

Note that if $g$ acts nontrivially on $V$, then $v_1$ and $v_2$
must span $(V^g)^\perp$ and $\alpha_g(v_1\wedge v_2)$ lies in $V^g$
as $\alpha_g(v_1\wedge v_2)$ is nonzero and $\alpha_g\in H^2_g$. 
Similarily, if $h$ acts nontrivially on $V$,
then $v_3$ and $\alpha_g(v_1\wedge v_2)$ must span $(V^h)^\perp$
and $w$ lies in $V^h$.
(See the comments after~(\ref{H-DOT}) 
or~\cite[Lemma~3.6]{SheplerWitherspoon1}.) 

Suppose first that {\em both} $g$ and $h$ act nontrivially on $V$.
Then $v_3$ and $\alpha_g(v_1\wedge v_2)$ are independent
vectors in $V^g\cap (V^h)^\perp$, a subspace of the 2-dimensional
space $(V^h)^\perp$.  Thus $(V^h)^\perp \subset V^g$
and $v_1, v_2$ in $(V^g)^\perp$ are fixed by $h$.
As $G$ is abelian and $\alpha$ is $G$-invariant, 
$\alpha_g= \ {^h}\alpha_g$ and
$$
\alpha_g(v_1\wedge v_2)
=(^{h^{-1}}\alpha_g)(v_1\wedge v_2)
=\ ^{h^{-1}}\Big(\alpha_g\big(\, ^hv_1\wedge\, ^hv_2\big)\Big)
=\ ^{h^{-1}}\Big(\alpha_g(v_1\wedge v_2)\Big)\ .
$$
But then $\alpha_g(v_1\wedge v_2)$
is fixed by $h$, contradicting the fact that it lies
in $(V^h)^\perp$.

We use the fact that the cochain map $[\alpha,\beta]$
represents the zero cohomology class
to analyze the case when either $g$ or $h$
acts trivially on $V$.
Calculations show that the image of the differential $d^*$
is supported on elements of $G$ that do not fix $V$ pointwise
(see, for example, Section~\ref{sec:PBW}).
Hence $V^{hg}\neq V$
and 
either $g$ or $h$ acts nontrivially on $V$.
Also note that the coefficient of $gh$
in any image of the differential lies in $(V^{gh})^\perp$.

If $h$ acts nontrivially on $V$ but $g$ fixes $V$ pointwise,
then $w$ lies in $(V^{gh})^\perp=(V^h)^\perp$, contradicting
the fact that $w$ lies in $V^h$ (as $h$ acts nontrivially).  
If instead $g$ acts nontrivially on $V$ but $h$ fixes $V$ pointwise,
we contradict the $G$-invariance of $\beta$: In this case,
$$
\begin{aligned}
w&=\beta_h\big(v_3\wedge\alpha_g(v_1\wedge v_2)\big)
= (^{g^{-1}}\beta_h)\big(v_3\wedge\alpha_g(v_1\wedge v_2)\big)\\
&=\ ^{g^{-1}}\Big(\beta_h\big(\,^gv_3\wedge\,^g(\alpha_g(v_1\wedge v_2))\big)\Big)
=\ ^{g^{-1}}\Big(\beta_h\big(v_3\wedge(\alpha_g(v_1\wedge v_2))\big)\Big)
=\ ^{g^{-1}} w
\end{aligned}
$$
(since both $v_3$ and $\alpha_g(v_1\wedge v_2)$ lie in $V^g$),
so $w$ lies in $V^g=V^{gh}$ instead of $(V^{gh})^\perp$.
\end{proof}


As a consequence of Lemma~\ref{prop:3condns} and Proposition~\ref{prop:abelian2},
we obtain the following: 

\begin{cor}
Let $G$ be an abelian group.
Suppose $\kappa^L$ in $(H^2)^G$ is a linear cocycle
with $[\kappa^L,\kappa^L]$ a coboundary. 
Then $[\kappa^L,\kappa^L]=0$ as a cochain.
Thus we
obtain a Drinfeld orbifold algebra $\cH_{\kappa}$
after setting $\kappa^C\equiv 0$ and $\kappa :=\kappa^L$.
\end{cor}

Other Drinfeld orbifold algebras with the same parameter
$\kappa^L$ arise from solving
the equation $[\kappa^C,\kappa^L] = 0$ for $\kappa^C$ a cocycle of
polynomial degree 0 in $(H^2)^G$. 
Compare with~\cite[Theorem~3.4]{HOT}, which is stated in the case that the
action is faithful.

We end this section by pointing out a much stronger statement than that implied
by~\cite[Theorem~9.2]{bracket} for abelian groups:
There we proved that for all groups $G$,
the bracket of any two Hochschild 2-cocycles supported off the kernel
of the representation is a coboundary (i.e., zero in cohomology).
The proposition below (cf.~\cite[Lemma~3.3]{HOT})
explains that when $G$ is abelian, such brackets are
not only coboundaries, they are zero as cochains.

\begin{prop}\label{lemma:abelian1}
Let $G$ be an abelian group. 
Let $\alpha,\beta$ in $(H^2)^G$ be two linear
Hochschild 2-cocycles on $S\# G$ supported off of the 
kernel of the representation
$G\rightarrow\text{GL}(V)$.
Then $[\alpha,\beta]=0$ as a cochain.
\end{prop}
\begin{proof}
This statement follows immediately from~\cite[Theorem~9.2]{bracket}
and Proposition~\ref{prop:abelian2}. However,
we give a short, direct proof here:
Let  $v_1,\ldots,v_n$ be a basis of $V$ on which $G$ acts diagonally.
If $[\alpha,\beta]$ is nonzero, 
then some summand 
$$
\beta_h\,
\left(v_3\wedge \alpha_g(v_1\wedge v_2)\right)
$$
of Definition~\ref{bracketdef} is nonzero for some
triple $v_1,v_2,v_3$ in $V$ and some $g$ and $ h$ in $G$.
Since $g$ and $h$ both act nontrivially
on $V$, the vector
$\alpha_g(v_1\wedge v_2)$ must be invariant under $h$
(as we saw in the third paragraph 
of the proof of Proposition~\ref{prop:abelian2}).
But this contradicts the fact that
$v_3$ and $\alpha_g(v_1\wedge v_2)$
must span $(V^h)^{\perp}$. 
\end{proof}

One may apply Proposition~\ref{lemma:abelian1} to find many examples
of Drinfeld orbifold algebras of the type given in Example~\ref{example:klein4group}.


\end{document}